\documentclass[english,ruled, 10pt]{article}
\usepackage[T1]{fontenc}
\usepackage[utf8]{inputenc}
\usepackage{amsfonts,amsmath, amssymb, amsthm}
\usepackage{geometry}
\usepackage{algorithm, algpseudocode}
\usepackage{nicematrix}
\usepackage{bm}
\usepackage{xcolor}
\usepackage{subcaption} 
\usepackage{diagbox}
\usepackage{tikz}
\usetikzlibrary{arrows.meta}
\usepackage{adjustbox}
\usepackage{lipsum}
\usepackage{graphicx} 
\usepackage{multirow}
\usepackage{booktabs}
\usepackage{optidef}
\usepackage{xcolor}
\usepackage{pstricks, pst-tree}
\usepackage[square,numbers]{natbib}
\usepackage{mathabx}
\usepackage{mathtools}
\usepackage{listings}
\usepackage{float}
\usepackage{caption}
\usepackage{pifont}
\usepackage{babel}
\usepackage{hyperref}
\hypersetup{
	colorlinks=true,
	linkcolor=red,
	citecolor = blue
}

\usepackage{changes}
\newcommand{\stkout}[1]{\ifmmode\text{\sout{\ensuremath{#1}}}\else\sout{#1}\fi}
\setdeletedmarkup{\stkout{#1}}
\definechangesauthor[name={Digvijay Boob}, color=blue]{DB}

\theoremstyle{plain}
\newtheorem*{assumption*}{\protect\assumptionname}

\theoremstyle{remark}
\newtheorem*{remark*}{\protect\remarkname}

\theoremstyle{plain}
\newtheorem{remark}{\protect\remarkname}

\theoremstyle{plain}
\newtheorem{theorem}{\protect\theoremname}%[section]

\theoremstyle{definition}
\newtheorem{definition}{\protect\definitionname}

\theoremstyle{plain}

\theoremstyle{plain}
\newtheorem{proposition}{\protect\propositionname}

\theoremstyle{plain}
\newtheorem{lemma}{\protect\lemmaname}

\theoremstyle{plain}

\providecommand{\assumptionname}{Assumption}
\providecommand{\corollaryname}{Corollary}
\providecommand{\definitionname}{Definition}
\providecommand{\lemmaname}{Lemma}
\providecommand{\propositionname}{Proposition}
\providecommand{\remarkname}{Remark}
\providecommand{\theoremname}{Theorem}

\newcommand*{\Scale}[2][4]{\scalebox{#1}{$#2$}}

\numberwithin{equation}{section}

\newenvironment{keywords}{
  \par\noindent
  \textbf{Keywords:}
}{
  \par
}

\newif\ifsiam

\siamfalse % or \siamtrue

\title{Regularized Operator Extrapolation Method For Stochastic Bilevel Variational Inequality Problems
}
\author{
Mohammad Khalafi\thanks{mohamadk@smu.edu, Operations Research and Engineering Management, Southern Methodist University}
\hspace{6em}
Digvijay Boob\thanks{dboob@smu.edu, Operations Research and Engineering Management, Southern Methodist University}
}
\date{}

\begin{document}
\allowdisplaybreaks
\global\long\def\vertiii#1{\left\vert \kern-0.25ex  \left\vert \kern-0.25ex  \left\vert #1\right\vert \kern-0.25ex  \right\vert \kern-0.25ex  \right\vert }%
\global\long\def\matr#1{\bm{#1}}%
\global\long\def\til#1{\tilde{#1}}%
\global\long\def\wt#1{\widetilde{#1}}%
\global\long\def\wh#1{\widehat{#1}}%
\global\long\def\wb#1{\widebar{#1}}%
\global\long\def\mcal#1{\mathcal{#1}}%
\global\long\def\mbb#1{\mathbb{#1}}%
\global\long\def\mtt#1{\mathtt{#1}}%
\global\long\def\ttt#1{\texttt{#1}}%
\global\long\def\inner#1#2{\langle#1,#2\rangle}%
\global\long\def\binner#1#2{\big\langle#1,#2\big\rangle}%
\global\long\def\Binner#1#2{\Big\langle#1,#2\Big\rangle}%
\global\long\def\br#1{\left(#1\right)}%
%\global\long\def\norm#1{\left\Vert #1\right\Vert }%
\global\long\def\bignorm#1{\bigl\Vert#1\bigr\Vert}%
\global\long\def\Bignorm#1{\Bigl\Vert#1\Bigr\Vert}%
\global\long\def\setnorm#1{\Vert#1\Vert_{-}}%
\global\long\def\rmn#1#2{\mathbb{R}^{#1\times#2}}%
\global\long\def\deri#1#2{\frac{d#1}{d#2}}%
\global\long\def\pderi#1#2{\frac{\partial#1}{\partial#2}}%
\global\long\def\onebf{\mathbf{1}}%
\global\long\def\zero{\mathbf{0}}%

\global\long\def\norm#1{\lVert#1\rVert}%
\global\long\def\bnorm#1{\big\Vert#1\big\Vert}%
\global\long\def\Bnorm#1{\Big\Vert#1\Big\Vert}%

% All kinds of brackets
\global\long\def\brbra#1{\big(#1\big)}%
\global\long\def\Brbra#1{\Big(#1\Big)}%
\global\long\def\rbra#1{(#1)}%
\global\long\def\sbra#1{[#1]}%
\global\long\def\bsbra#1{\big[#1\big]}%
\global\long\def\Bsbra#1{\Big[#1\Big]}%
\global\long\def\cbra#1{\{#1\}}%
\global\long\def\bcbra#1{\big\{#1\big\}}%
\global\long\def\Bcbra#1{\Big\{#1\Big\}}%

% expectation
\global\long\def\grad{\nabla}%
\global\long\def\Expe{\mathbb{E}}%
\global\long\def\rank{\text{rank}}%
\global\long\def\range{\text{range}}%
\global\long\def\diam{\text{diam}}%
\global\long\def\epi{\text{epi }}%
\global\long\def\inte{\operatornamewithlimits{int}}%
\global\long\def\cov{\text{Cov}}%
\global\long\def\argmin{\operatornamewithlimits{argmin}}%
\global\long\def\argmax{\operatornamewithlimits{argmax}}%
\global\long\def\tr{\operatornamewithlimits{tr}}%
\global\long\def\dis{\operatornamewithlimits{dist}}%
\global\long\def\sign{\operatornamewithlimits{sign}}%
\global\long\def\prob{\mathbb{P}}%
\global\long\def\st{\operatornamewithlimits{s.t.}}%
\global\long\def\dom{\text{dom}}%
\global\long\def\diag{\text{diag}}%
\global\long\def\and{\text{and}}%
\global\long\def\st{\text{s.t.}}%
\global\long\def\Var{\operatornamewithlimits{Var}}%
\global\long\def\raw{\rightarrow}%
\global\long\def\law{\leftarrow}%
\global\long\def\Raw{\Rightarrow}%
\global\long\def\Law{\Leftarrow}%
\global\long\def\vep{\varepsilon}%
\global\long\def\dom{\operatornamewithlimits{dom}}%

\global\long\def\Lbf{\mathbf{L}}%

\global\long\def\Ffrak{\mathfrak{F}}%
\global\long\def\Gfrak{\mathfrak{G}}%
\global\long\def\gfrak{\mathfrak{g}}%
\global\long\def\Hfrak{\mathfrak{H}}%
\global\long\def\Ofrak{\mathfrak{O}}%
\global\long\def\sfrak{\mathfrak{s}}%
\global\long\def\vfrak{\mathfrak{v}}%
\global\long\def\xibar{\bar{\xi}}%
\global\long\def\Cbb{\mathbb{C}}%
\global\long\def\Ebb{\mathbb{E}}%
\global\long\def\Fbb{\mathbb{F}}%
\global\long\def\Nbb{\mathbb{N}}%
\global\long\def\Rbb{\mathbb{R}}%
\global\long\def\extR{\widebar{\mathbb{R}}}%
\global\long\def\Pbb{\mathbb{P}}%
\global\long\def\Acal{\mathcal{A}}%
\global\long\def\Bcal{\mathcal{B}}%
\global\long\def\Ccal{\mathcal{C}}%
\global\long\def\Dcal{\mathcal{D}}%
\global\long\def\Fcal{\mathcal{F}}%
\global\long\def\Gcal{\mathcal{G}}%
\global\long\def\Hcal{\mathcal{H}}%
\global\long\def\Ical{\mathcal{I}}%
\global\long\def\Kcal{\mathcal{K}}%
\global\long\def\Lcal{\mathcal{L}}%
\global\long\def\Mcal{\mathcal{M}}%
\global\long\def\Ncal{\mathcal{N}}%
\global\long\def\Ocal{\mathcal{O}}%
\global\long\def\Pcal{\mathcal{P}}%
\global\long\def\Ucal{\mathcal{U}}%
\global\long\def\Scal{\mathcal{S}}%
\global\long\def\Tcal{\mathcal{T}}%
\global\long\def\Xcal{\mathcal{X}}%
\global\long\def\Ycal{\mathcal{Y}}%
\global\long\def\Ubf{\mathbf{U}}%
\global\long\def\bt{\mathbf{t}}
\global\long\def\bx{\mathbb{x}}
\global\long\def\Pbf{\mathbf{P}}%
\global\long\def\Ibf{\mathbf{I}}%
\global\long\def\Ebf{\mathbf{E}}%
\global\long\def\Abs{\boldsymbol{A}}%
\global\long\def\Qbs{\boldsymbol{Q}}%
\global\long\def\Lbs{\boldsymbol{L}}%
\global\long\def\Pbs{\boldsymbol{P}}%
\global\long\def\gap{\text{Gap}}%
\global\long\def\dist{\text{dist}}%
\newcommand{\proj}{\textbf{proj}}
\newcommand{\prox}{\textbf{prox}}
\newcommand{\proximal}{\text{prox}}
\def\bx{\mathbf{x}}
\global\long\def\i{i}%
\global\long\def\Ibb{\mathbb{I}}

%\global\long\def\inprod#1#2{\left\langle #1,#2\right\rangle }%
\DeclarePairedDelimiterX{\inprod}[2]{\langle}{\rangle}{#1, #2}
\DeclarePairedDelimiter\abs{\lvert}{\rvert}
\DeclarePairedDelimiter{\bracket}{ [ }{ ] }
\DeclarePairedDelimiter{\paran}{(}{)}
\DeclarePairedDelimiter{\braces}{\lbrace}{\rbrace}
\DeclarePairedDelimiterX{\gnorm}[3]{\lVert}{\rVert_{#2}^{#3}}{#1}
\DeclarePairedDelimiter{\floor}{\lfloor}{\rfloor}
\DeclarePairedDelimiter{\ceil}{\lceil}{\rceil}

\global\long\def\tsum{{\textstyle {\sum}}}

\newcommand{\opconex}{\texttt{OpConEx}}
\newcommand{\sopconex}{\texttt{S-OpConEx}}
\newcommand{\stochsopconex}{\texttt{S-StOpConEx}}
\newcommand{\adopconex}{\texttt{AdLagEx}}
\newcommand{\gconex}{\texttt{GradConEx}}
\newcommand{\stopconex}{\texttt{StOpConEx}}
\newcommand{\fstopconex}{\texttt{F-StOpConEx}}
\newcommand{\aconex}{\texttt{Aug-ConEx}}
\newcommand{\augconex}{\text{Aug-ConEx}}
\newcommand{\iropex}{\texttt{R-OpEx}}
\newcommand{\textiropex}{\text{IR-OpEx}}
\maketitle

\begin{abstract}
The bilevel variational inequality (BVI) problem is a general model that captures various optimization problems, including VI-constrained optimization and equilibrium problems with equilibrium constraints (EPECs). 
    This paper introduces a first-order method for smooth or nonsmooth BVI with stochastic monotone operators at inner and outer levels. Our novel method, called Regularized Operator Extrapolation $(\iropex)$, is a single-loop algorithm that combines Tikhonov's regularization with operator extrapolation. This method needs only one operator evaluation for each operator per iteration and tracks one sequence of iterates. We show that $\iropex$ gives $\mathcal{O}(\epsilon^{-4})$ complexity in nonsmooth stochastic monotone BVI, where $\epsilon$ is the error in the inner and outer levels. Using a mini-batching scheme, we improve the outer level complexity to $\mathcal{O}(\epsilon^{-2})$ while maintaining the  $\mathcal{O}(\epsilon^{-4})$ complexity in the inner level when the inner level is smooth and stochastic. Moreover, if the inner level is smooth and deterministic, we show complexity of $\mathcal{O}(\epsilon^{-2})$. Finally, in case the outer level is strongly monotone, we improve to $\mathcal{O}(\epsilon^{-4/5})$ for general BVI and $\mathcal{O}(\epsilon^{-2/3})$ when the inner level is smooth and deterministic. To our knowledge, this is the first work that investigates nonsmooth stochastic BVI with the best-known convergence guarantees. We verify our theoretical results with numerical experiments. 
 \vspace{-2mm}
% \noindent {\bf Keywords:} Variational Inequality, Function Constraints, Stochastic first-order methods, Saddle-point problems with coupled constraints
\end{abstract}

\vspace{-1mm}
% In this paper, we present novel stochastic first-order methods which can handle function constrained VI (FCVI) problems under deterministic, stochastic (operator only), and fully-stochastic (both stochastic operator and function constraints) cases where both the operator and function constraints can be either smooth, nonsmooth, or the composition of smooth and nonsmooth components. We show convergence rates/sample complexities for the proposed algorithms 
\section{Introduction}
%\vspace{-3mm}
The variational inequality (VI) problem associated with the monotone operator $\Tilde{F}$ over the convex set $\Tilde{X}$ is formally given by $\text{VI}(\tilde{F},\Tilde{X})$ as below
\begin{equation}\label{eq:VI}
    \text{Find } x^* \in \Tilde{X}: \quad \inprod{\Tilde{F}(x^*)}{x - x^*} \geq 0, \quad \forall x \in \Tilde{X}.
\end{equation} 
In this paper, we use the standard definition of VI in \eqref{eq:VI} and construct the following monotone bilevel variational inequality problem (BVI)
\begin{equation}\label{eq:bilevel-vi}
\begin{aligned}
    &\text{Find } && x^* \in X^*_F: &&\quad \inprod{H(x^*)}{x - x^*} \geq 0, &&\quad \forall x \in X^*_F, \\
    &\text{where } && x^*_F \in X: &&\quad \inprod{F(x^*_F)}{y - x^*_F} \geq 0, &&\quad \forall y \in X,
\end{aligned}
\end{equation}
where $X\subseteq \mathbb{R}^n$ is a closed convex set, $H:X\rightarrow \mathbb{R}^n$ and $F:X\rightarrow \mathbb{R}^n$ are possibly discontinuous monotone operators. We denote \textit{outer} and \textit{inner} VI problems as $\text{VI}(H,X^*_F)$ and $\text{VI}(F,X)$, respectively. BVIs have attracted a lot of interest as they encompass various classes of problems, such as the optimization
problem with variational inequality constraints, EPECs, complementarity problems, and optimal selection problems\cite{luo1996mathematical,mordukhovich2004equilibrium,beck2014first,kaushik2021method}. Considering both levels to be data-driven, i.e., the operators have the form of $F(x) =\Ebb_\xi[F(x,\xi)]$ and $H(x) = \Ebb_\zeta[H(x,\zeta)]$ for some random variables $\xi,\zeta$ defined on appropriate probability spaces, the problems above have numerous applications in machine learning such as model selection \cite{kunapuli2008classification}, hyperparameter tuning \cite{franceschi2018bilevel,feurer2019hyperparameter,lorraine2020optimizing}, lexicographic optimization \cite{jiang2023conditional}, meta-learning \cite{kim2025stochastic}, and sparse feature selection\cite{poon2021smooth}. There are some key challenges regarding BVIs. 
First, it is hard to project on the feasible set $X^*_F$ of the outer VI problem in many cases. Therefore, popular projection-based approaches, including projected gradient, extragradient (EG) \cite{korpelevich1976extragradient, nemirovski2004prox}, and optimistic gradient descent ascent \cite{Mokhtari2020}, are not applicable directly to the VI$(H, X^*_F)$. % since these methods require projection onto the solution set of the inner level. In fact, it appears no algorithm outputs the entire solution set \cite{feinstein2024characterizing}. 
Second, consider $H(x) = \grad f(x)$ for some convex function $f$; thus \eqref{eq:bilevel-vi} reduces to a VI-constrained convex minimization problem. However, unlike the usual bilevel optimization problems where we have finitely many functional constraints at the inner level, the corresponding constraint set is characterized by infinitely many inequalities. Hence, we cannot apply methods based on Lagrangian relaxation and  Karush–Kuhn–Tucker (KKT) conditions as in functional-constrained problems. In the following, we present a literature review on VIs and BVIs. 
\\
\textbf{Related work:}
\textbf{Variational inequalities.} The VI problem was first introduced by \citet{minty1962monotone} and \citet{stampacchia1964formes}. Since then, VIs have played a vital role in different engineering problems \cite{facchinei2003finite,pang2010design,shanbhag2013stochastic}. One popular solution method for the usual VIs is EG introduced by \citet{korpelevich1976extragradient} and developed in \cite{nemirovski2004prox,censor2011subgradient,iusem2011korpelevich,juditsky2011solving,iusem2017extragradient,chen2017accelerated}. In a seminal work, \citet{nemirovski2004prox} introduced a generalized EG-type algorithm called mirror-prox and showed $\mathcal{O}(\epsilon^{-1})$ complexity for smooth and deterministic VIs which improved the existing $\mathcal{O}(\epsilon^{-2})$ rate in projected gradient methods. This is the optimal complexity based on the lower bounds of first-order methods in bilinear saddle-point problems provided by \citet{ouyang2021lower}. Further, \citet{nesterov2006solving} obtained linear convergence for the strongly monotone VI problem. In addition, more simplistic algorithms such as projected reflected gradient methods \cite{malitsky2015projected} and operator extrapolation \cite{kotsalis2022simple} were used for solving monotone VIs. Recently, an operator constraint extrapolation method \cite{boob2024firstordermethodsstochasticvariational} was introduced for solving VI problems with complex functional constraints. Unlike EG methods, these algorithms need only one projection onto the ``simple" set $X$ in each iteration and track a single sequence of iterates. The VI problem becomes extremely challenging in the stochastic setting where we cannot access the operator's exact evaluation. \citet{jiang2008stochastic} established stochastic approximation (SA) to obtain the asymptotic convergence for the stochastic VIs with monotone and smooth operators. Later, papers on merely monotone and non-Lipschitz stochastic VIs were done based on the SA method \cite{juditsky2011solving,koshal2012regularized,yousefian2017smoothing}. The extended version of the mirror-prox method exhibited worst-case complexity of $\mathcal{O}(\epsilon^{-2})$ \cite{juditsky2011solving}. \citet{yousefian2014optimal} achieved $\mathcal{O}(\epsilon^{-1})$ complexity for stochastic VIs under weak sharpness assumption. \citet{boob2024firstordermethodsstochasticvariational} proved the optimal complexity of $\mathcal{O}(\epsilon^{-2})$ for VI problems with stochastic operator and expectation function constraints.
\\
\textbf{Bilevel variational inequalities.} The existing literature on BVIs and VI-constrained optimization problems is scarce. Only a few contributions have been made \cite{xu2004viscosity,yamada2011minimizing,facchinei2014vi,yousefian2017smoothing,kaushik2021method,kaushik2023incremental,samadi2025improved}. \citet{yamada2011minimizing} proposed a \textit{hybrid steepest descent method} (HSDM) to solve a nonsmooth bilevel optimization problem. \citet{xu2004viscosity} presented a sequential averaging method on the BVI problem. However, neither of those works has provided any convergence rate in inner or outer levels. Few works on BVIs and VI-constrained optimization provide explicit complexity results for optimality (outer-level) and feasibility (inner-level) gaps. \citet{kaushik2021method}  showed  $\mathcal{O}(\epsilon^{-4})$ convergence rate for optimality and feasibility gaps for a convex VI-constrained optimization problem. In the follow-up work, \citet{samadi2025improved} improved the rate to $\mathcal{O}(\epsilon^{-2})$ for a deterministic smooth BVI problem. To our knowledge, \cite{samadi2025improved} is the most recent work that constructs optimality and feasibility gaps for a bilevel variational inequality problem.
\\
\textbf{Regularization and gradient-based methods.} As discussed before, BVIs are difficult to solve given that (i) it is difficult to project onto the feasible set $X^*_F$, and (ii) $X^*_F$ consists of infinitely many functional constraints. Indeed, it implies the Lagrangian dual variable must have infinite dimension, which is impractical. A simple idea to address these challenges stems from the celebrated Tikhonov's regularization method introduced in \cite{tikhonov1963solution} where one uses iterative methods to solve the inner and outer problems in \eqref{eq:bilevel-vi} simultaneously. In the context of BVIs, Tikhonov's regularization uses an alternative operator ${O}(\cdot,\eta):=F(\cdot) + \eta H(\cdot) $  with regularization parameter $\eta>0$. \citet{tikhonov1963solution} showed that one can solve a `simple' bilevel optimization problem by obtaining the optimal solution of ${O}(\cdot,\eta^*)$ over the set $X$. However, since $\eta^*$ is not known, one can  use ${O}(\cdot,\eta_k)$ for a sequence $\{\eta_k\}_{k\geq 1}$ satisfying $\eta_k \rightarrow 0$ . Using $\eta_k \rightarrow 0$ and $\tsum_{k=1}^\infty \eta_k = \infty$, \citet{cabot2005proximal} and \citet{solodov2007explicit} proposed a projected gradient method and showed asymptotic convergence. Motivated by regularization-based approaches, some works combined standard first-order methods, such as block coordinate or extragradient, with an iteratively regularized scheme to solve deterministic BVI\cite{kaushik2021method,samadi2025improved}. A few papers address more general bilevel optimization problems. \citet{solodov2007bundle} used a bundle method for nonsmooth bilevel convex optimization problems to obtain asymptotic convergence. Recently, \citet{doron2023methodology} developed the ITerative Approximation
and Level-set EXpansion (ITALEX) approach with explicit convergence rates for inner and outer levels.\\
\textbf{Contributions:}
In this paper, we introduce a new method based on operator extrapolation \cite{kotsalis2022simple,boob2024firstordermethodsstochasticvariational} and Tikhonov's regularization for a general class of bilevel variational inequality problems where both operators $H(\cdot)$ and $F(\cdot)$ are nonsmooth and stochastic. We summarize our contributions below. 

\textbf{First}, we provide an explicit and best-known convergence rate for optimality and feasibility gaps for stochastic nonsmooth BVIs. Among the handful of works that measure suboptimality and infeasibility of the solutions \cite{kaushik2021method,kaushik2023incremental,samadi2025improved}, none of them study a fully-stochastic BVI problem with nonsmooth operators.
Specifically, we show $\mathcal{O}(\epsilon^{-4})$ oracle complexity for a general nonsmooth BVI.\\
\textbf{Second}, in case we have a smooth stochastic operator $F$, we show $\mathcal{O}(\epsilon^{-2})$ complexity in operator $H$ by doing mini-batching in operator $F$ and maintaining $\mathcal{O}(\epsilon^{-4})$ complexity for $F$. We further improve the overall complexity of $\mathcal{O}(\epsilon^{-2})$ given that the operator $F$ is smooth and deterministic.\\
\textbf{Third}, considering operator $H$ to be strongly monotone, we improve to $\mathcal{O}(\epsilon^{-4/5})$ in the general case and $\mathcal{O}(\epsilon^{-2/3})$ in the case with smooth and deterministic $F(\cdot)$. \\
\textbf{Finally}, unlike the state-of-the-art methods such as \cite{samadi2025improved}, we do not need additional operator evaluation, which is inevitable in the EG-based approaches. This is inherently rooted in the operator extrapolation method. 
%\end{itemize}
In Table \ref{tab:comparison}, we compare the most relevant methods in BVIs, namely \cite{kaushik2021method,kaushik2023incremental,samadi2025improved} with our algorithm regarding the used assumptions and settings and convergence rates.
\begin{table}[h!]
\centering
%\vspace{-5mm}
\caption{Comparison of algorithms under various settings and convergence rates.}
  \begin{adjustbox}{width=\textwidth}
\begin{tabular}{|c|cc|cc|cc|cc|}
\hline
\multirow{3}{*}{\textbf{Algorithm}} 
& \multicolumn{4}{c|}{\textbf{Setting}} 
& \multicolumn{4}{c|}{\textbf{Operator Complexity}} \\
\cline{2-9}
& \multicolumn{2}{c|}{Stochastic Operator} 
& \multicolumn{2}{c|}{Nonsmooth Operator} 
& \multicolumn{2}{c|}{Monotone} 
& \multicolumn{2}{c|}{Strongly Monotone} \\
\cline{2-9}
& $F$ & $H$ & $F$ & $H$ 
& $F$ & $H$ & $F$ & $H$ \\
\hline
$\text{aRB-IRG}^{\ast}$\cite{kaushik2021method} & $\times$ & $\times$ & $\times$ & $\times$ 
& $\mathcal{O}(1/\epsilon^4)$ & $\mathcal{O}(1/\epsilon^4)$
& $-$ & $-$ \\
$\text{pair-IG}^{\dagger}$\cite{kaushik2023incremental} & $\times$ & $\times$ & $\times$ & $\times$ 
& $\mathcal{O}(1/\epsilon^4)$ & $\mathcal{O}(1/\epsilon^4)$
& $-$ & $-$ \\
$\text{IR-EG}^{\ddagger}$\cite{samadi2025improved} & $\times$ & $\times$ & $\times$ & $\times$ 
& $\mathcal{O}(1/\epsilon^2)$ & $\mathcal{O}(1/\epsilon^2)$
& $\mathcal{O}(\ln(1/\epsilon)/\epsilon)$ & $\mathcal{O}(\ln(1/\epsilon)/\epsilon)$ \\
$\iropex$ & $\checkmark$ & $\checkmark$ & $\checkmark$ & $\checkmark$ 
& $\mathcal{O}(1/\epsilon^4)$ & $\mathcal{O}(1/\epsilon^4)$
& $\mathcal{O}(1/\epsilon^{\tfrac{4}{5}})$ & $\mathcal{O}(1/\epsilon^{\tfrac{4}{5}})$ \\
$\iropex$ (Mini-baching) & $\checkmark$ & $\checkmark$ & $\times$ & $\checkmark$ 
& $\mathcal{O}(1/\epsilon^4)$ & $\mathcal{O}(1/\epsilon^2)$
& $-$ & $-$ \\
$\iropex$ & $\times$ & $\checkmark$ & $\times$ & $\checkmark$ 
& $\mathcal{O}(1/\epsilon^2)$ & $\mathcal{O}(1/\epsilon^2)$
& $\mathcal{O}(1/\epsilon^{\tfrac{2}{3}})$ & $\mathcal{O}(1/\epsilon^{\tfrac{2}{3}})$ \\
\hline
\end{tabular}
\end{adjustbox}
\label{tab:comparison}
\begin{minipage}{\textwidth} % Adjust width to match your table
\footnotesize
\raggedright % aligns text inside minipage to the left
$^\ast$ Averaging randomized block iteratively regularized gradient.\\
$^\dagger$ Projected averaging regularized incremental subgradient. \\
$^\ddagger$ Iteratively regularized extragradient.
\end{minipage}
\end{table}
%\vspace{-7mm}
\paragraph{Outline:}
In Section \ref{sec:prem}, we discuss the notation, preliminaries and assumptions.Section \ref{sec:algorithm} comprehensively describes $\iropex$ in detail to solve problem \eqref{eq:bilevel-vi}. Section \ref{sec:convegence-analysis-iropex} provides a unified convergence analysis of the $\iropex$ method for different cases, such as general BVI or BVI with a smooth inner level. We illustrate numerical experiments in Section \ref{sec:numerical} and mention the conclusion and research discussion in Section \ref{sec:conclude}. 
%\vspace{-3mm}
\section{Notation and preliminaries}\label{sec:prem}%\vspace{-3mm}
%\subsection{Notation}\label{sec:notation}
We use the following notation in this paper. We denote $[m]:\{1,\cdots, m\}$. A vector $x\in \mathbb{R}^n$ is a column vector, and its transpose is shown as $x^\top$. Furthermore, $\mathbb{R}^n_+$ represents the non-negative orthant of an $n$-dimensional Euclidean space. Euclidean is denoted as $\gnorm{\cdot}{}{}$ and the standard inner product is defined as $\inprod{\cdot}{\cdot}$. We say the operator $F:X\rightarrow \mathbb{R}^n$ is monotone on the convex set $X\subseteq \mathbb{R}^n$ if $\inprod{F(x)-F(y)}{x-y}\geq 0$ for any $x,y \in X$. Moreover, the operator $F$ is said to be $\mu_F$-strongly monotone if $\inprod{F(x)-F(y)}{x-y}\geq \mu\gnorm{x-y}{}{}$ for any $x,y\in X$. We denote Euclidean projection of a vector onto a closed convex set $X$ as $\proj_X(x)$, where $\proj_X(x) := \argmin_{y\in X} \gnorm{x-y}{}{}$. Consequently, we define the distance function $\text{dist}(x,X)$ for a given vector $x$ and closed convex set $X$ as $\text{dist}(x,X): = \min_{y\in X}\gnorm{x-y}{}{} = \gnorm{x- \proj_X(x)}{}{}$. The operator $F$ is said to be $L_F$-\textit{Lipschitz smooth} if $\gnorm{F(x) - F(y)}{}{} \le L_F\gnorm{x-y}{}{} $, for all $x,y\in X$. We will delve into the idea of smoothness and nonsmoothness later. We define the diameter of a compact set $X$ by $D_X : = \max_{x,y\in X} \tfrac{\gnorm{x-y}{}{}}{2}$.
%\subsection{Preliminaries}\label{sec:premil}\\
Now, we discuss definitions and assumptions used throughout this paper. First, we denote the solution set of \eqref{eq:bilevel-vi} as $X_H^*$ and assume it is nonempty. Next, we consider the general case where $F$ and $H$ are possibly composed of Lipschitz continuous and bounded discontinuous monotone operators, satisfying
\begin{subequations}
    \begin{equation}\label{eq:F-Lipschitz-property}
	\gnorm{F(x_1) - F(x_2)}{}{} \le L_F\gnorm{x_1-x_2}{}{} + M_F, \quad \forall x_1, x_2\in X,
\end{equation}
\begin{equation}\label{eq:H-Lipschitz-property}
    \gnorm{H(x_1) - H(x_2)}{}{} \le L_H\gnorm{x_1-x_2}{}{} + M_H, \quad \forall x_1, x_2\in X,
\end{equation}
\end{subequations}
where the $L_F,L_H$ terms stem from the Lipschitz continuous components and the $M_F,M_H$ correspond to the discontinuous components of operators $F,H$, respectively. If $M_F=0$ or $M_H=0$, the corresponding operator is Lipschitz continuous, which is known as a "smooth operator" in the VI literature. Consequently, one can classify the bilevel VI problem in \eqref{eq:bilevel-vi} into smooth and nonsmooth cases. The problem \eqref{eq:bilevel-vi} is a \textit{smooth BVI} if both operators $F$ and $H$ are smooth, i.e., $M_F = M_H=0$. If at least one of $M_F$ or $M_H$ is positive, then \eqref{eq:bilevel-vi} is a \textit{nonsmooth BVI} problem - the main focus of this paper. Apart from smoothness and nonsmoothness assumptions, we say the operators are bounded by $C_F$ and $C_H$ in the set $X$ and by $B_F$ and $B_H$ for the set $X^*_F$ as follows 
\begin{subequations}
    \begin{equation}\label{eq:boundedness-F}
    \gnorm{F(x)}{}{}\leq C_F, \forall x\in X, \quad \gnorm{F(x)}{}{}\leq B_F, \forall x\in X_F^*,
    \end{equation}
    \begin{equation}\label{eq:boundedness-H}
    \gnorm{H(x)}{}{}\leq C_H, \forall x\in X, \quad \gnorm{H(x)}{}{}\leq B_H, \forall x\in X_F^*.
    \end{equation}
\end{subequations}
\paragraph{Convergence criteria.}
In this paper, we define three criteria to evaluate any obtained solutions from Algorithm \ref{alg:IRopex}. The first two measures correspond to the gap functions for each inner and outer VI problem. First, let us define the inner gap function or feasibility gap. 
\begin{definition}\label{def:feasib-gap}
    Let $X\subseteq \mathbb{R}^n$ be a nonempty, closed, and convex set and $F: X\rightarrow \mathbb{R}^n$ be a monotone operator. Then we say that $\Tilde{x}$ is an $\epsilon_f$-solution regarding the feasibility gap function associated with $\text{VI}(F,X)$ if 
    \begin{equation}\label{eq:feasib-gap}
        \gap(\Tilde{x},F,X):=\max_{x\in X} \inprod{F(x)}{\Tilde{x}-x}\leq \epsilon_f,
    \end{equation}
Similarly, we say $\Tilde{x}$ is a stochastic $\epsilon_f$-solution regarding the feasibility gap function if 
    \begin{equation*}\label{eq:feasib-gap-stoch}
        \Ebb[\gap(\Tilde{x},F,X)]:=\Ebb[\max_{x\in X} \inprod{F(x)}{\Tilde{x}-x}]\leq \epsilon_f,
    \end{equation*}
\end{definition}
%\vspace{-3mm}
Furthermore, one can define the outer gap function, also known as the optimality gap, as follows. 
\begin{definition}\label{def:optim-gap}
    Let $X\subseteq \mathbb{R}^n$ be a nonempty, closed, and convex set, and $H:X\rightarrow \mathbb{R}^n$ be a monotone operator. Then we say that $\Tilde{x}$ is an $\epsilon_o$-solution regarding the optimality gap function associated with $\text{VI}(H,X^*_F)$ if 
    \begin{equation}\label{eq:optim-gap}
        \gap(\Tilde{x},H,X^*_F):=\max_{x\in X^*_F} \inprod{H(x)}{\Tilde{x}-x}\leq \epsilon_o,
    \end{equation}
    Similarly, we say $\Tilde{x}$ is a stochastic $\epsilon_o$-solution regarding the optimality gap function if 
     \begin{equation*}\label{eq:optim-gap=stoch}
        \Ebb[\gap(\Tilde{x},H,X^*_F)]:=\Ebb[\max_{x\in X^*_F} \inprod{H(x)}{\Tilde{x}-x}]\leq \epsilon_o,
    \end{equation*}
\end{definition}
%\vspace{-3mm}
Henceforth, we maintain $\epsilon_f$ and $\epsilon_o$ of the same order and commonly refer to them as  $\epsilon$.
 Note that $\gap(\Tilde{x},F,X) = 0$ (respectively, $\gap(\Tilde{x},H,X^*_F) = 0$) if and only if $\Tilde{x}$ is a solution to $\text{VI}(F,X)$ (respectively, $\text{VI}(H,X^*_F)$). Moreover, it is easy to see $\gap(\Tilde{x}, F, X) \geq 0$ for all $\Tilde{x}\in X$ and $\gap(\Tilde{x}, H, X_F^*)\geq 0$ for all $\Tilde{x}\in X_F^*$.  However, if $\Tilde{x}\notin X$, then one might obtain $\gap(\Tilde{x},F,X) \leq 0 $. Similarly, if $\Tilde{x}\notin X_F^*$ then we might have  $\gap(\Tilde{x},H,X^*_F)\leq 0$. Due to the structure of our proposed method, the last case where we have negative values in optimality is likely to happen. Therefore, we also need to obtain lower bounds for the optimality gap. \\
The last measure of convergence is related to the distance between the solutions generated by Algorithm \ref{alg:IRopex} from the feasible set $X_F^*$. In particular, we are interested in the distance function $\text{dist}(\tilde{x}, X_F^*)$ where $\tilde{x} \in X$ is a given point.

{\bf Stochastic approximation of operators.}
We assume the operators $F$ and $H$ are in expectation form. We refer to this as a stochastic version of the problem \eqref{eq:bilevel-vi} as a fully-stochastic bilevel VI. We use stochastic oracles (SO) to generate random vectors estimating $F$ and $H$. In particular, given a random variable $\xi$ which is independent of the search point $x$, SOs generate random vectors of $\Ffrak (x,\xi)$ and $\Hfrak (x,\xi)$ such that 
\begin{subequations}
    \begin{equation}\label{eq:stoch-oracle-F}
        \forall x\in X, \quad F(x) = \Ebb[\Ffrak(x,\xi)],\quad \Ebb[\gnorm{F(x)-\Ffrak(x,\xi)}{}{2}]\leq \sigma_F^2,
    \end{equation}
    \begin{equation}\label{eq:stoch-oracle-H}
        \forall x\in X, \quad H(x) = \Ebb[\Hfrak(x,\xi)],\quad \Ebb[\gnorm{F(x)-\Hfrak(x,\xi)}{}{2}]\leq \sigma_H^2.
    \end{equation}
\end{subequations}
 %Basically, we call $\Ffrak (x,\xi)$ and $\Hfrak (x,\xi)$ the unbiased estimators of $F$ and $H$, respectively. 
 \paragraph{Weak sharpness of operators.}
 Throughout this paper, we rely on the weak sharpness assumption to establish some error bounds. 
\begin{definition}\label{def:weak-sharp}
         We say that $\text{VI}(F, X)$ is $\alpha$-weakly sharp if there exists a scalar $\alpha > 0$ such that for all $x \in X$ and $x^* \in X^*_F$, the following holds
         \begin{equation*}\label{eq:weakly-sharp}
             \inprod{F(x^*)}{x-x^*}\geq \alpha \text{dist}(x,X^*_F). 
         \end{equation*}
         \end{definition}

\section{The Regularized operator extrapolation method } \label{sec:algorithm}
%\vspace{-10pt}
In the $k$-th iteration of the $\iropex$ method with the search point $x_k$, we generate random vectors $\Ffrak_k:=\Ffrak(x_k,\xi_k)$ and $\Hfrak_k:=\Hfrak(x_k,\xi_k)$ using the SOs we define in \eqref{eq:stoch-oracle-F} and \eqref{eq:stoch-oracle-H}. Unlike the usual extragradient methods, $\iropex$ is a simple algorithm without intermediate steps. The most crucial idea in $\iropex$ is that it incorporates regularization techniques with operator extrapolation methods. Specifically, we extrapolate the regularized operator $O(\cdot,\eta_k):=F(\cdot) + \eta_k H(\cdot) $ with the extrapolation parameter $\theta_k$. Note that the regularization parameter $\eta_k$ plays a weighting role in how much information is used from each operator. Usually, to get convergence in the regularization method, we need $\tsum_{k=1}^\infty\eta_k =\infty$ and $\eta_k\rightarrow 0$. Moreover, $\gamma_k$ denotes the step-size for the proximal operator and $\tau_k$ is the averaging parameter to obtain $\bar{x}_K $. See Algorithm \ref{alg:IRopex} for more details. 
\begin{algorithm}[H]
	\caption{\textbf{R}egularized \textbf{Op}erator \textbf{Ex}trapolation (\iropex) method}\label{alg:IRopex}
	\begin{algorithmic}[1]
		\State {\bf Input:} $x_0 = x_1 \in X$, $\{\tau_1,\eta_0,\eta_{1},\gamma_1, \theta_1\}$, $K$
		\For{k = $1, \dots, K-1$}
		\begin{equation}\label{eq:x-update}
				x_{k+1} \gets\argmin_{x \in X} \inprod{\Ffrak_k+\eta_k\Hfrak_k + \theta_k\big[\Ffrak_k+\eta_{k-1}\Hfrak_k - [\Ffrak_{k-1}+\eta_{k-1}\Hfrak_{k-1}]\big]}{x} + \tfrac{1}{2\gamma_k}\gnorm{x-x_k}{}{2}
		\end{equation}
		\EndFor\\
		\Return $\bar{x}_K = \tsum_{k=0}^{K-1}\tfrac{\tau_k x_{k+1}}{\tau_k}$.
	\end{algorithmic}
\end{algorithm} 
\section{Convergence analysis of R-OpEx method} \label{sec:convegence-analysis-iropex}
This section provides an integrated convergence analysis for the $\iropex$ method. We consider two main subcategories to study the convergence rate. First, we provide convergence and complexity results for the general BVI with nonsmooth and stochastic operators. Second, we assume the inner problem $\text{VI}(F,X)$ is smooth and mentions the corresponding rates. In the following lemma, we provide a recursive inequality for each iteration of Algorithm \ref{alg:IRopex}.
 \begin{lemma}\label{lem:one-iteration-iropex}
     Consider the problem \eqref{eq:bilevel-vi} with $ \mu_H\geq 0$. Let $\{x_k\}_{k\geq 1}$ be a sequence generated by Algorithm \ref{alg:IRopex}. Then we have the following relation for any $x\in X$
     \begin{equation}\label{eq:one-iteration-iropex}
     \Scale[.95]{
     \begin{aligned}
         &\inprod{F(x)+ \eta_kH(x)}{x_{k+1}-x}\leq \tfrac{1}{2\gamma_k}[\gnorm{x-x_{k}}{}{2}-\gnorm{x_{k+1}-x_{k}}{}{2}] - (\tfrac{1}{2\gamma_k} + \eta_k\mu_H)\gnorm{x-x_{k+1}}{}{2}\\
         &+ \inprod{\Delta\Ofrak_{k+1}}{x_{k+1}-x} - \theta_k\inprod{\Delta\Ofrak_k}{x_k-x}
         + \theta_k (L_F + \eta_{k-1} L_H)\gnorm{x_k-x_{k-1}}{}{}\gnorm{x_{k+1}-x_{k}}{}{}\\
         &-\inprod{\delta_{k+1}^F + \eta_k \delta_{k+1}^H}{x_{k+1}-x}- \theta_k \inprod{\delta_k^F -\delta_{k-1}^F+ \eta_{k-1}[\delta_k^H -\delta_{k-1}^H]}{x_{k+1}-x_k}\\
          &+ \theta_k(M_F + \eta_{k-1}M_H)\gnorm{x_{k+1}-x_k}{}{},
     \end{aligned}}
     \end{equation}
     where $\Delta \Ofrak_k = \Ffrak_k - \Ffrak_{k-1} + \eta_{k-1}[\Hfrak_k-\Hfrak_{k-1}] $ and $\delta_k^F = \Ffrak_k - F(x_k)$, $\delta_k^F = \Hfrak_k - H(x_k)$.
     \end{lemma}
     %\vspace{-2mm}
     Now, we are ready to state the results for the general nonsmooth and stochastic problem.
     %\vspace{-3mm}
     \subsection{General nonsmooth and stochastic BVI}\label{sec:general-Bilevel}
     %\vspace{-3mm}
     In this section, we consider the most general case where we assume both operators $F$ and $H$ are in nonsmooth and stochastic forms. Notably, we assume $M_F, \sigma_F$ and $M_H, \sigma_H$ have positive values. We discuss the convergence results for the monotone case in Theorems \ref{thm:convex-case-optim-feasib}-\ref{thm:lower-bound-mono} and discuss the similar results for the strongly monotone case in Section \ref{sec:strong-mono}.  
     \begin{theorem}\label{thm:convex-case-optim-feasib}
         Let us assume problem \eqref{eq:bilevel-vi} is monotone $( \mu_H = 0)$. Suppose we have the following step-size policy 
         \begin{equation}\label{eq:step-size-con-bivel}
         \Scale[.95]{
             \tau_k =\tau =1,\quad \theta_k =\theta=1,\quad \eta_k= \eta= K^{-\tfrac{1}{4}},\quad \gamma_k = \gamma =\tfrac{D_X}{8D_X(L_F + \eta L_H)+ \sqrt{K(M_F^2 +2\sigma_F^2+\eta^2 (M_H^2+2\sigma_H^2))}},}
         \end{equation}
         Then, Algorithm \ref{alg:IRopex} gives the following optimality and feasibility upper bounds for $K\geq 1$
         \begin{equation}\label{eq:mono-optim}
\Scale[.93]{
\begin{aligned}
    &-B_H\Ebb[\text{dist}(\bar{x}_K,X^*_F)]\leq \Ebb[\gap(\bar{x}_K, H, X^*_F)] \leq  D_X\bigg[16D_X(\tfrac{L_F}{K^{3/4}} +\tfrac{L_H}{K}) + \tfrac{2\sqrt{M_F^2+2\sigma_F^2+\eta^2(M_H^2+ 2\sigma_H^2)}}{K^{1/4}} \\
    &+\tfrac{5(M_F^2+2\sigma_F^2+\eta^2(M_H^2+2\sigma_H^2)) + \sigma_F^2 + \eta^2\sigma_H^2}{8D_XK^{-1/4}(L_F + \eta L_H) + K^{1/4}\sqrt{M_F^2+2\sigma_F^2+\eta^2(M_H^2+2\sigma_H^2)}}+ \tfrac{5(M_F^2+2\sigma_F^2+\eta^2(M_H^2+2\sigma_H^2))}{8D_XK^{3/4}(L_F + \eta L_H) + K^{5/4}\sqrt{M_F^2+2\sigma_F^2+\eta^2(M_H^2+\sigma_H^2)}}\bigg].
     \end{aligned}}
\end{equation}
%\vspace{-2mm}
and 
\begin{equation}\label{eq:mono-feasib}
\Scale[.9]{
    \begin{aligned}
        &\Ebb[\gap(\bar{x}_K,F,X)]\leq D_X\bigg[16D_X(\tfrac{L_F}{K} +\tfrac{L_H}{K^{5/4}}) + \tfrac{2\sqrt{M_F^2+2\sigma_F^2+\eta^2(M_H^2+2\sigma_H^2)}}{K^{1/2}} 
        \\&+\tfrac{5(M_F^2+2\sigma_F^2+\eta^2(M_H^2+2\sigma_H^2))+ \sigma_F^2 + \eta^2\sigma_H^2}{8D_X(L_F + \eta L_H) + K^{1/2}\sqrt{M_F^2+2\sigma_F^2+\eta^2(M_H^2+2\sigma_H^2)}}
         + \tfrac{5(M_F^2+2\sigma_F^2+\eta^2(M_H^2+\sigma_H^2))}{8D_XK(L_F + \eta L_H) + K^{3/2}\sqrt{M_F^2+2\sigma_F^2+\eta^2(M_H^2+2\sigma_H^2)}}+\tfrac{2 C_H}{K^{1/4}}\bigg].
    \end{aligned}}
\end{equation}
     \end{theorem}
     %\vspace{-3mm}
     \begin{remark}
        Note that using step-size policy in \eqref{eq:step-size-con-bivel}, Algorithm \ref{alg:IRopex} gives overall rate of $\mathcal{O}\big(D_X(\tfrac{D_XL_F}{K^{3/4}}+\tfrac{D_XL_H}{K}+\tfrac{M_F +\sigma_F} {K^{1/4}}+\tfrac{M_H +\sigma_H }{K^{1/2}} )\big)$.  This is the best-known rate for a stochastic and nonsmooth  BVI problem.
     \end{remark}
     Next, observe that constructing a lower bound on the optimality gap is crucial, given that Algorithm \ref{alg:IRopex} uses a single projection onto set $X$ and not set $X_F^*$. Therefore, it is also possible that we obtain a negative optimality gap. In Theorem \ref{thm:lower-bound-mono}, we construct an explicit lower bound for the optimality gap. 
     \begin{theorem}\label{thm:lower-bound-mono}
           Suppose $\text{VI}(F,X)$ is $\alpha$-weakly sharp. Then we have 
           \begin{equation}\label{eq:lower-bound-mono}
           \Scale[0.9]{
           \begin{aligned}
               & \Ebb[\gap(\bar{x}_K, H, X^*_F)] \geq  -\tfrac{B_HD_X }{\alpha}\bigg[16D_X(\tfrac{L_F}{K} +\tfrac{L_H}{K^{5/4}}) + \tfrac{2\sqrt{M_F^2+2\sigma_F^2+\eta^2(M_H^2+2\sigma_H^2)}}{K^{1/2}}
        \\&+\tfrac{5(M_F^2+2\sigma_F^2+\eta^2(M_H^2+2\sigma_H^2))+ \sigma_F^2 + \eta^2\sigma_H^2}{8D_X(L_F + \eta L_H) + K^{1/2}\sqrt{M_F^2+2\sigma_F^2+\eta^2(M_H^2+2\sigma_H^2)}}
         + \tfrac{5(M_F^2+2\sigma_F^2+\eta^2(M_H^2+\sigma_H^2))}{8D_XK(L_F + \eta L_H) + K^{3/2}\sqrt{M_F^2+2\sigma_F^2+\eta^2(M_H^2+2\sigma_H^2)}}+\tfrac{2 C_H}{K^{1/4}}\bigg].
           \end{aligned}}
           \end{equation}
     \end{theorem}
     %\vspace{-3mm}
     Additionally, we provide explicit oracle complexity in case the threshold $\tfrac{\alpha}{\gnorm{H(x^*)}{}{}}$ is available. Remark \ref{rem:alpha-H} states the convergence rate in the presence of $\tfrac{\alpha}{\gnorm{H(x^*)}{}{}}$. For more details, the reader is referred to Theorem \ref{thm:weakly-sharp-conv} in Section \ref{sec:analysis-monotone}. 
     \begin{remark}\label{rem:alpha-H}
         Consider we change $\eta = K^{-1/4}$ to a value satisfying $\eta\leq\tfrac{\alpha}{\gnorm{H(x^*)}{}{}}$ and keep the rest of parameters as \eqref{eq:step-size-con-bivel}. Then Algorithm \ref{alg:IRopex} gives rate of $\mathcal{O}\big( D_X ( \tfrac{D_XL_F}{K} + \tfrac{D_XL_H}{K^{5/4}} + \tfrac{M_F + M_H + \sigma_F + \sigma_H}{K^{1/2}} ) \big)$. This improves the overall rate from $\mathcal{O}(1/K^{1/4})$ to $\mathcal{O}(1/K^{1/2})$.
     \end{remark}
     %\vspace{-2mm}
     Next, we provide the convergence rates for the strongly monotone case, where $\mu_H > 0$. Section~\ref{sec:strong-mono} presents a complete analysis of this case.
     \begin{remark}\label{rem:strongly-mono}
         Suppose the problem \eqref{eq:bilevel-vi} is strongly monotone $(\mu_H>0)$. Consider $\tau_k = k+1, \theta_k = \tfrac{k}{k+1}$ and keep $\eta_k, \gamma_k$ as \eqref{eq:step-size-con-bivel}. Then, we obtain $\mathcal{O}\big( D_X ( \tfrac{D_XL_F}{K^{7/4}} + \tfrac{D_XL_H}{K^2} + \tfrac{M_F + \sigma_F}{K^{5/4}} + \tfrac{M_H + \sigma_H}{K^{3/2}} ) \big)$. Further, suppose $\tfrac{\alpha}{\gnorm{H(x^*)}{}{}}$ is available then the overall convergence rate is improved to $\mathcal{O}\big( D_X ( \tfrac{D_XL_F}{K^2} + \tfrac{D_XL_H}{K^{9/4}} + \tfrac{M_F + M_H + \sigma_F + \sigma_H}{K^{3/2}} ) \big)$. 
     \end{remark}
     %\vspace{-2mm}
     As one can observe from \eqref{eq:step-size-con-bivel}, we should have an estimate of the number of iterations \( K \) a priori. To address this, we propose a new step-size policy that eliminates the need to know the number of iterations in advance.
     %\vspace{-2mm}
     \subsubsection{Step-size policy for $\eta_k$ independent of $K$}\label{sec:adaptive-main}
     %\vspace{-2mm}
     This section specifies a new step-size policy independent of $K$. The full analysis is in Section \ref{sec:adaptive}. While we provide analysis of such a policy for only the nonsmooth fully stochastic case, similar claims can be made for the remaining cases. We drop their discussion for the sake of brevity.
     \begin{remark}
        Consider the following adaptive updating policy
        \begin{equation*}\label{eq:adap-step-size-main}
        \Scale[.9]{
            \tau_k = \tau = 1, \quad \eta_k = (k+1)^{-1/4}, \quad \theta_k = (\tfrac{k}{k+1})^{1/4}, \quad \gamma_k = \tfrac{D_X}{8D_X (L_F + \eta_k L_H) + \sqrt{k(M_F^2 + 2\sigma_F^2 + \eta_k^2(M_H^2+2\sigma_H^2))}}}.
        \end{equation*}
    Then  Algorithm \ref{alg:IRopex} gives $\mathcal{O}\big(D_X(\tfrac{D_XL_F}{K^{3/4}}+\tfrac{D_XL_H}{K}+\tfrac{M_F +\sigma_F} {K^{1/4}}+\tfrac{M_H +\sigma_H }{K^{1/2}} )\big)$ convergence rate.
     \end{remark}
     %\vspace{-2mm}
     Moreover, similar to Theorem \ref{thm:lower-bound-mono}, we provide an explicit lower bound for the optimality gap in Section \ref{sec:adaptive}.
     %\vspace{-3mm}
     \subsection{Smooth inner VI}
     %\vspace{-3mm}
    In this section, we analyze the convergence of Algorithm \ref{alg:IRopex} under the assumption that the operator \( F \) is smooth, i.e., \( M_F = 0 \). This setting can be broken into two subcategories depending on whether we have stochasticity in \( F \) $(\sigma_F >0)$ or not $(\sigma_F = 0)$. In case of stochastic \( F \), we use mini-batching to reduce the effect of $\sigma_F$. In particular, for size $B$ mini-batching to the operator $\Ffrak$, we have $\bar{\Ffrak}_B(x) = \tfrac{1}{B}\tsum_{i=1}^{B}\Ffrak(x,\xi_i)$ to approximate operator operator $F$ in line two of Algorithm \ref{alg:IRopex}. Remark \ref{rem:smooth-batch} states the overall convergence rate and operator complexity for an inner-smooth BVI with stochastic or deterministic $F$. The comprehensive analysis for the smooth inner VI case is presented in Section \ref{sec:smooth-VI-appden}.
    \begin{remark}\label{rem:smooth-batch}
        {\bf Stochasitc $F$:} Assume that problem \eqref{eq:bilevel-vi} is monotone $(\mu_H = 0)$. Additionally, let $M_F = 0$, and suppose we have the following step-size policy with mini-batching of size $B = K$
         %\vspace{-2mm}
         \begin{equation*}\label{eq:step-size-con-bivel-smooth-stoch-F-rem}
         \Scale[.95]{
             \tau_k =\tau =1,\quad \theta_k =\theta=1,\quad \eta_k= \eta= K^{-\tfrac{1}{2}},\quad \gamma_k = \gamma =\tfrac{D_X}{8D_X(L_F +\eta L_H)+ \sqrt{M_H^2+2(\sigma_H^2 + \sigma_F^2)}},}.
         \end{equation*}
         Then, $\iropex$ gives the overall convergence rate of $\mathcal{O}\big(D_X(\tfrac{D_XL_F}{K^{1/2}} + \tfrac{D_XL_H}{K} + \tfrac{M_H+\sigma_F+\sigma_H}{K^{1/2}})\big)$.\\
         {\bf Deterministic $F$:} If we have a smooth and deterministic inner VI $(M_F = \sigma_F = 0)$, then we obtain the same convergence rate as above. Finally, if $(M_F = \sigma_F = 0)$ and suppose BVI is strongly monotone $(\mu_H>0)$, then, the overall convergence rate improves to $\mathcal{O}\big(D_X(\tfrac{D_XL_F}{K^{3/2}} + \tfrac{D_XL_H}{K^2} + \tfrac{M_H+\sigma_F+\sigma_H}{K^{3/2}})\big)$.
    \end{remark}
    %\vspace{-2mm}
    Note that we improve from convergence rate $\mathcal{O}(K^{-1/4})$ in the general nonsmooth BVI to $\mathcal{O}(K^{-1/2})$. Although it keeps the oracle complexity as $\mathcal{O}(\epsilon^{-4})$ for the stochastic $F$ given that we do mini-batching, we improve in terms of oracle complexity in $H$ to $\mathcal{O}(\epsilon^{-2})$. Note that in the strongly monotone case for stochastic $F$, no mini-batching scheme on $F$ improves upon the current $\mathcal{O}(\epsilon^{-4/5})$ complexity in $F$ and $H$. %, we keep perform a single sample evaluation of $F$.
    Finally, results for deterministic $F$ imply that the \iropex~maintains $\mathcal{O}(\epsilon^{-2})$ complexity in $H$ when $F$ is deterministic. The complexity can be improved to $\mathcal{O}(\epsilon^{-2/3})$ if $\mu_H > 0$.
%\vspace{-3mm}
%\input{sec-conclude}

% \section{Saddle point with coupling constraints}
% Consider the general class of saddle point problems
% \begin{equation}
%     \begin{aligned}
%         \min_u \max_v & \ f(u,v) \\
%         \text{s.t.}\   & \ g(u,v) \le 0
%     \end{aligned}
% \end{equation}

% Define $\widetilde{W}_u(v)=\{u: g(u,v)\le 0\}$ and $\widetilde{W}_v(u)=\{v: g(u,v)\le 0\}$.
%  Then $(\hat{u}, \hat{v})\in U\times V$ is a Generalized Nash Equilibrium of this problem if 
% \[
% \max_{v\in V\cap \widetilde{W}_{\hat{u}}} f(\hat{u}, v) \le f(\hat{u}, \hat{v})\le \min_{u\in U\cap \widetilde{W}_{\hat{v}}} f(u, \hat{v}).
% \]

% This motivates us to use the following criterion:
% We say that $(\hat{u}, \hat{v})\in U\times V$ is a $\epsilon$-Generalized Nash Equilibrium if $g(\hat{u}, \hat{v}) \le \epsilon$  and 
% \begin{equation}
%     \begin{aligned}
%         \max_{u\in U} &\   f(\hat{u},\hat{v})-f(u,\hat{v}) & \le \epsilon\\
%         \text{s.t. }  & \ g(u,\hat{v})\le \epsilon &  \\
%         \max_{v\in V} & \ f(\hat{u},v)-f(\hat{u},\hat{v}) & \le \epsilon \\
%         \text{s.t. } & \ g(\hat{u}, v)\le \epsilon
%     \end{aligned}
% \end{equation}
\section{Numerical experiments}\label{sec:numerical}
%\vspace{-5mm}
In this section, we validate the results from Section \ref{sec:convegence-analysis-iropex}, particularly Theorem \ref{thm:convex-case-optim-feasib} and Remark \ref{rem:strongly-mono}, on two distinct settings. First, we test the performance of Algorithm $\iropex$ on a stochastic convex optimization problem whose feasible region is defined on a stochastic two-player zero-sum Nash game. Next, we test our algorithm on a traffic equilibrium problem with uncertainty in the inner or outer problems. All the experiments are implemented on 64-bit Windows 11 with Intel i7-1260P @ 2.10GHz and 16GB RAM. 
%\vspace{-5mm}
\paragraph{VI-constrained optimization problem.}
Consider the following optimization problem 
\begin{equation}\label{eq:VI_cons_num}
    \min_{x\in X^*_F}\psi(x):=\Ebb[ \tfrac{1}{2}\gnorm{x+ \zeta}{}{2}]
\end{equation}
where $\zeta \in \mathbb{R}^2$ is a random vector whose elements $\zeta_i, i=1,2$ are independent and identically distributed following $\mathcal{N}(0,1)$. Moreover, considering the set $ X  = [20,50]\times [5,15]$, we define the feasible region $X^*_F\subseteq X$ as the solution set of the following Nash equilibrium (NE) problem 
\begin{equation}\label{eq:VI_cons_inner}
   \min_{x_1\in [20,50]}\max_{x_2\in [5,15]} f(x_1,x_2):= \Ebb[25 - 2x_1x_2 + \xi x_1],
\end{equation}
where the random variable $\xi$ is following $\mathcal{N}(10,1)$. One can notice that the explicit form of $X^*_F$ can be stated as $X^*_F: = \{(x_1,x_2)| x_1 \in  [20,50], x_2 = 5\}$.  Furthermore, considering the explicit form of $X^*_F$, the minimum of \eqref{eq:VI_cons_num} is obtained at $x^* = (20,5)$. Note that the corresponding operators $\Ffrak$ and $\Hfrak$ have the following expressions 
\[
\Ffrak(x_1,x_2;\xi)=\begin{bmatrix}
-2x_2 + \xi\\
2x_1 
\end{bmatrix}
\quad
\Hfrak(x_1,x_2;\zeta)=\begin{bmatrix}
x_1 + \zeta_1 \\
x_2 + \zeta_2 
\end{bmatrix}
\]
We implement Algorithm \ref{alg:IRopex} on problem \eqref{eq:VI_cons_num} for monotone and strongly monotone cases. Note that it is a smooth problem with $L_F = \gnorm{[0,-2;2,0]}{}{}$ and $L_H =1$.  We set the number of iterations $K = 5\time 10^6$ for $R = 10$ i.i.d replications. Figure \ref{fig:full_stoch_conv} depicts the optimality and feasibility gaps of the solutions generated by $\iropex$ regarding the number of iterations and time in the monotone setting. For the optimality and feasibility gaps, we use $\psi(\Bar{x}_k) - \psi(x^*)$ and $f(\Bar{x}_{1,k}, x^*_2) - f( x^*_1, \Bar{x}_{2,k})$, respectively. Each gap is plotted against the number of iterations and the time in seconds. Figure \ref{fig:full_stoch_conv_str} shows similar results for the strongly monotone case of implementing Algorithm \ref{alg:IRopex} on \ref{eq:VI_cons_num}. As expected, we have better performance in run time and accuracy of $\iropex$ in the strongly monotone setting. 
%\vspace{-5mm}
\begin{figure}[H]
    \centering
    \includegraphics[width=.8\textwidth]{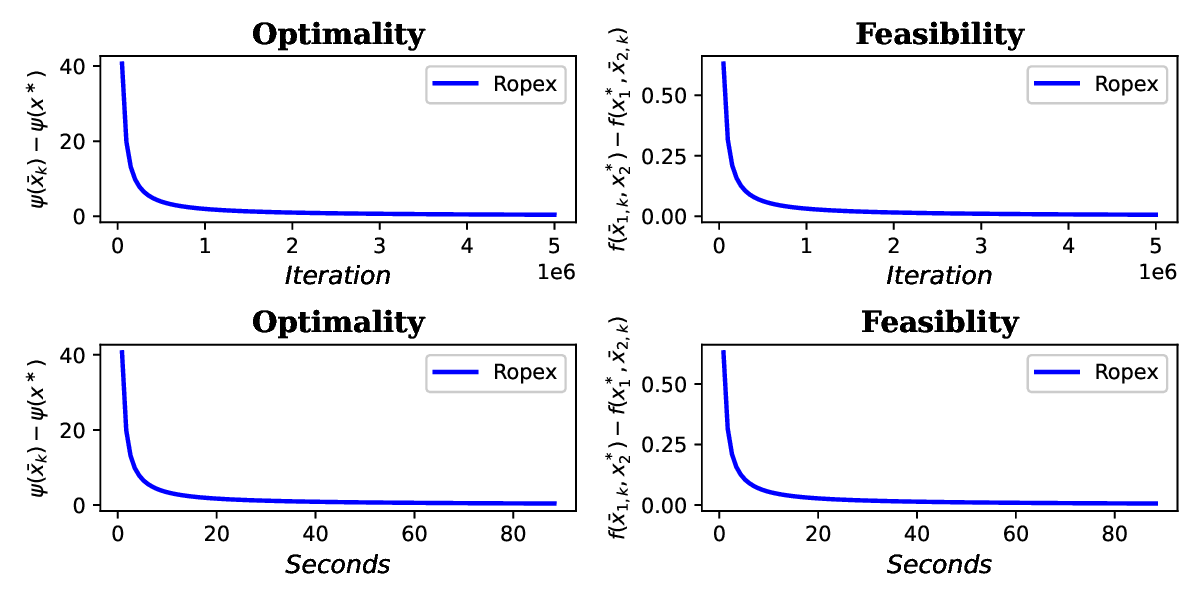}
    %\vspace{-4mm}
    \caption{Optimality and feasibility gaps for $10$ replication of Algorithm \ref{alg:IRopex} with $5\times 10^6$ iterations (monotone case)}
    \label{fig:full_stoch_conv}
\end{figure}
%\vspace{-4mm}
\begin{figure}[H]
    \centering
    \includegraphics[width=0.8\textwidth]{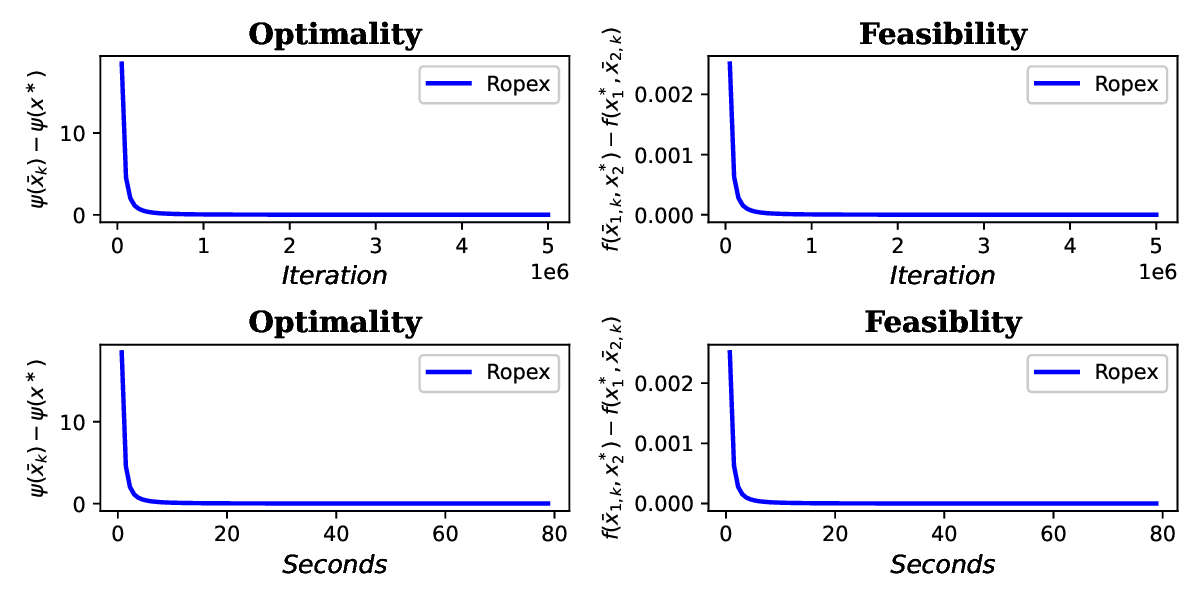}
    %\vspace{-4mm}
    \caption{Optimality and feasibility gaps for $10$ replication of Algorithm \ref{alg:IRopex} with $5\times 10^6$ iterations (strongly monotone case)}
    \label{fig:full_stoch_conv_str}
\end{figure}
%\vspace{-5mm}
\paragraph{Traffic equilibrium problem}
As shown in Figure \ref{fig:traffic}, we consider a traffic flow network.
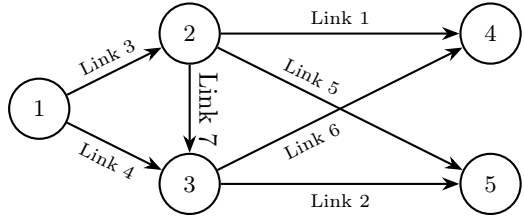
\begin{figure}[hbt!]
\centering
\begin{tikzpicture}[->, thick, node distance=1cm, auto, 
    mynode/.style={circle, draw, minimum size=0.8cm, font=\small},
    every path/.style={draw, ->, >=Stealth}
]
% Nodes
\node[mynode] (1) at (0,1) {1};
\node[mynode] (2) at (2,2) {2};
\node[mynode] (3) at (2,0) {3};
\node[mynode] (4) at (6,2) {4};
\node[mynode] (5) at (6,0) {5};
% Links with plain text labels
\draw (2) -- (4) node[midway, above] {\scriptsize Link 1};
\draw (3) -- (5) node[midway, below] {\scriptsize Link 2};
\draw (1) -- (2) node[midway, above, sloped] {\scriptsize Link 3};
\draw (1) -- (3) node[midway,below, sloped] {\scriptsize Link 4};
\draw (2) -- (5) node[midway, above left, sloped] {\scriptsize Link 5};
\draw (3) -- (4) node[midway, below left, sloped] {\scriptsize Link 6};
\draw (2) -- (3) node[midway, above, sloped] {Link 7};
\end{tikzpicture}
\caption{Traffic Network with 5 nodes and 7 links}
\label{fig:traffic}
\end{figure}
As illustrated, the traffic network we are interested in has 5 nodes, 7 links, 2 origin-destination (O-D) pairs (1$\rightarrow$ 4, 1$\rightarrow$ 5), and 6 paths $p_1 = \{3,7,6\}, p_2 = \{3,1\}, p_3 = \{4,6\}, p_4 = \{3,7,2\}, p_5 = \{3,5\}, p_6 = \{4,2\}$. Travel equilibrium problems are modeled according to travel demand between each (O-D) pair and capacity on each link. Let $d = [d_1 + \xi_1, d_2 + \xi_2]$ be the stochastic travel demand between (O-D) pairs where $\xi_i, i = 1,2$ is the noise following $\mathcal{N}(0,1)$. Also, denote $cap = [cap_1,\dots, cap_7]$ as the capacity vector associated to network links. The corresponding traffic flows for paths and links are denoted by $h = [h_1,\dots, h_6]$ and $f = [f_1, \dots, f_7]$. Furthermore, $\Delta$ represents link-path incident matrix, where $\delta_{\ell, p}=1$ if path $p$ contains link $\ell$, and $\delta_{\ell, p}=0$ otherwise. Similarly, $\Omega$ denotes (O-D)-path matrix. 
In these experiments, we follow a generalized bureau of public roads (GBPR) to compute the link travel time function \cite{YIN2009470} as follows 
\[
c_\ell(f_\ell) = t_\ell^0(1+0.15(\tfrac{f_\ell}{cap_\ell})^{n_\ell}),
\]
where $n_\ell \geq 1$, and $t_\ell^0$ are given parameters. One can see $f = \Delta h$ implying the following 
\begin{equation}\label{eq:path-cost}
    C(h) = \Delta^\top c(\Delta h) 
\end{equation}
Let $u = [u_1,u_2]$ be the minimum travel costs between each (O-D) pair. Consequently, for a realization of the random vector $\xi =[\xi_1, \xi_2]$, one can use Wardrop's user equilibrium for the decision variable $x = [h,u]$ and construct the complementarity problem $x\geq 0, F(x)\geq 0,$ and $x^\top F(x) =0$ where we have the following operator $F \in \mathbb{R}^8$
\begin{equation*}
    F(x;\xi) = 
\begin{bmatrix}
C(h) - \Omega^\top u \\
\Omega h -d 
\end{bmatrix}
\end{equation*}
We set $n_\ell = 1 \forall \ell = 1,\dots, 7$, yielding a linear complementarity problem (LCP). Moreover, we use total travel cost of the network $\psi(x) = \Ebb[\zeta^\top C(h)]$ as the function we want to minimize in traffic equilibrium where $\zeta$ is a vector with each element $\zeta_i \sim \mathcal{U}_{0,2},i\in[6]$. Given that $\psi$ is convex and $F$ is a monotone operator for $n_\ell = 1, \ell\in [7]$ (see Section 6.2 of \cite{samadi2025improved}), we can formulate the problem as a stochastic VI constrained minimization where the outer problem is associated with minimization of $\psi(x)$ and the inner problem is $\textbf{VI}(F,\mathbb{R}^8_+)$. For the experiment setting, we set the capacity of each link $\ell$ to $cap_\ell = 400$. Let the stochastic travel demand between (O-D) be $d = [200+\xi_1, 220+\xi_2]$. Moreover, we choose $t_\ell^0 = 1,\forall \ell \in [7]$. We use $\gnorm{\Bar{x}_{k+1} - \Bar{x}_k}{}{}$ to measure suboptimality where $\Bar{x}_k $ is the solution generated by $\iropex$. Since the inner problem is an LCP, we use the function $\phi(x) = \gnorm{\min\{x,0\}}{}{} + \gnorm{\min\{F(x),0\}}{}{} + \gnorm{x^\top F(x)}{}{}$ to measure the infeasibility. Figures \ref{fig:full_stoch_conv_set2} and \ref{fig:full_stoch_str_conv_set2} demonstrate the performance of Algorithm \ref{alg:IRopex} for $R=10$ i.i.d instances and $K =5\times 10^6$ iterations in monotone and strongly monotone settings, respectively. Similar to the previous section, we plot the gaps in terms of the number of iterations and time in seconds.  
%\vspace{-4mm}
\begin{figure}[H]
    \centering
    \includegraphics[width=0.8\textwidth]{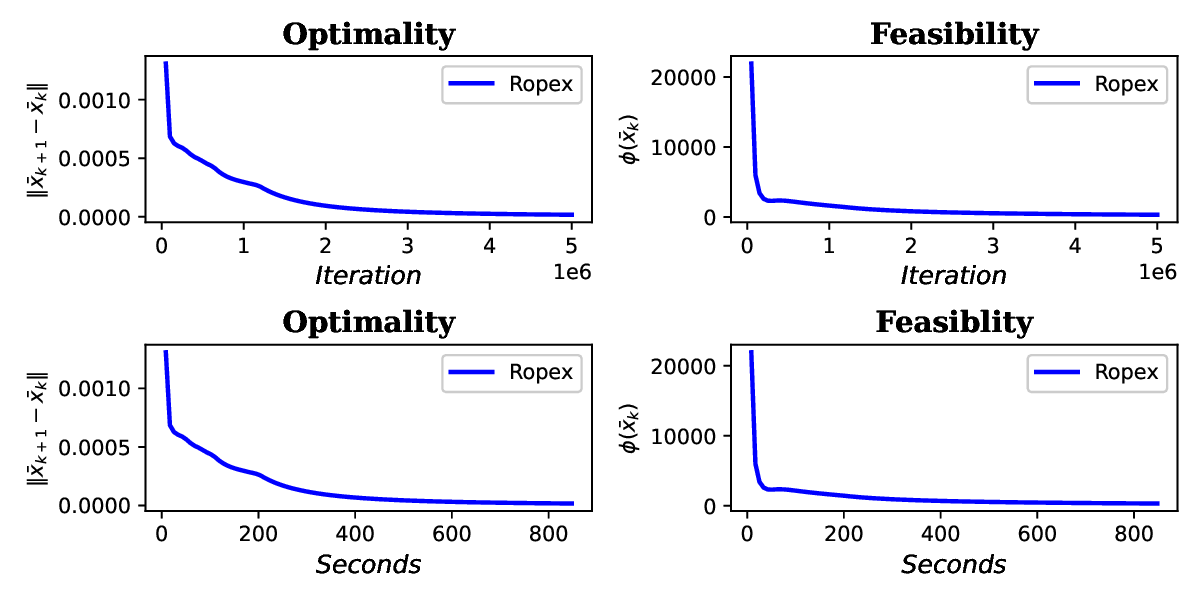}
    %\vspace{-8mm}
    \caption{Optimality and feasibility gaps for $10$ replication of Algorithm \ref{alg:IRopex} with $5\times 10^6$ iterations (monotone case)}
    \label{fig:full_stoch_conv_set2}
\end{figure}
%\vspace{-8mm}
\begin{figure}[H]
    \centering
    \includegraphics[width=0.8\textwidth]{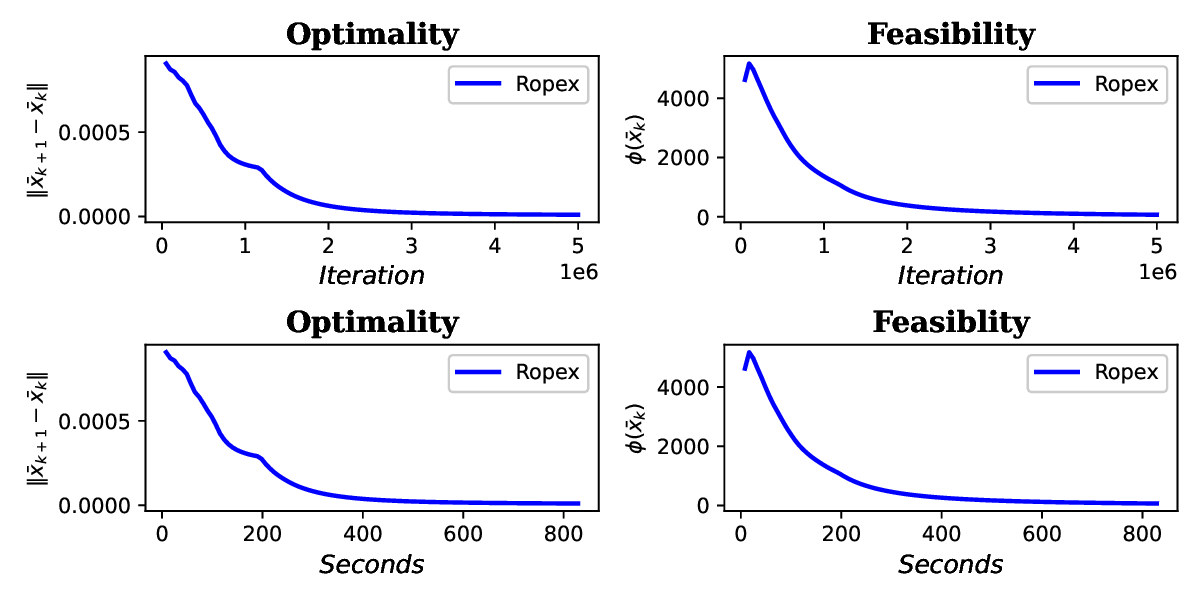}
    %\vspace{-8mm}
    \caption{Optimality and feasibility gaps for $10$ replication of Algorithm \ref{alg:IRopex} with $5\times 10^6$ iterations (strongly monotone case)}
    \label{fig:full_stoch_str_conv_set2}
\end{figure}
%\vspace{-8mm}
\section{Conclusion and discussion}\label{sec:conclude} \vspace{-3mm} This paper studies bilevel variational inequality problems with nonsmooth and stochastic operators. We propose a novel and easy-to-implement algorithm that utilizes the operator extrapolation and regularization methods and achieves the best known explicit convergence rates in both levels of BVI. These convergence rate guarantees are obtained using fixed or variable step-size policies, showcasing our method's flexibility. In addition, we improve the convergence rate and operator complexity if we have a smooth inner problem. Our proposed scheme matches the best existing convergence rates in the literature, whereas no work assumes stochasticity of the operators. Many interesting research directions exist, such as stochastic VI-constrained nonconvex optimization or last iterate methods.
\newpage
\appendix

\section{Appendix}
\subsection{Analysis of Lemma \ref{lem:one-iteration-iropex}}
First, we must state a key lemma known as the "three-point" lemma, which is crucial in the proof of Lemma \ref{lem:one-iteration-iropex}. 
 
 \begin{lemma}\label{lem:three-point}
     (See the proof in \cite{boob2023stochastic}) Let $x^*$ be a solution of problem $\min_{x\in X}\{h(x) + \tfrac{\lambda}{2}\gnorm{x-\hat{x}}{}{2}\} $ where $h(x)$ is a convex function. Then, 
     \begin{equation}\label{eq:three-point}
         h(x^*) - h(x) \leq \tfrac{\lambda}{2}[\gnorm{x-\hat{x}}{}{2} - \gnorm{x^*-x}{}{2}- \gnorm{x^*-\hat{x}}{}{2}]\quad \forall x\in X. 
     \end{equation}
 \end{lemma}
\subsubsection{Proof of Lemma \ref{lem:one-iteration-iropex}}
\begin{proof}
     From optimality of $x_{k+1}$ and Lemma \ref{lem:three-point}, we have the following 
     \begin{equation*}
    \begin{aligned}
         \inprod{\Ffrak_{k}+ \eta_k\Hfrak_{k}}{x_{k+1}-x}\leq &\tfrac{1}{2\gamma_k}[\gnorm{x-x_{k}}{}{2}-\gnorm{x-x_{k+1}}{}{2}-\gnorm{x_{k+1}-x_{k}}{}{2}]  - \theta_k\inprod{\Delta \Ofrak_k}{x_{k+1}-x},
     \end{aligned}
     \end{equation*}
Then from the definition of $\Delta \Ofrak_k$ we have 
 \begin{equation}\label{eq:recur-stoch}
 \begin{aligned}
      \inprod{\Ffrak_{k+1}+ \eta_k\Hfrak_{k+1}}{x_{k+1}-x}\leq &\tfrac{1}{2\gamma_k}[\gnorm{x-x_{k}}{}{2}-\gnorm{x-x_{k+1}}{}{2}-\gnorm{x_{k+1}-x_{k}}{}{2}]\\
     &+ \inprod{\Delta \Ofrak_{k+1}}{x_{k+1}-x} - \theta_k\inprod{\Delta \Ofrak_k}{x_{k}-x}
       -\theta_k\inprod{\Delta \Ofrak_k}{x_{k+1}-x_k}
 \end{aligned}
     \end{equation}
     Note that from definition of $\delta_k^F$ and $\delta_k^H$, and from strong monotonicity of $H$ with modules $\mu_H$, we have the following relations 
         \begin{equation}\label{eq:delta-last-error}
         \Scale[.8]{
         \begin{aligned}
             \inprod{\Ffrak_{k+1} + \eta_k \Hfrak_{k+1}}{x_{k+1}-x} &= \inprod{F(x_{k+1}) + \eta_k H_{k+1}}{x_{k+1}-x} + \inprod{\delta_{k+1}^F + \eta_k \delta_{k+1}^H}{x_{k+1}-x}\\
             &\geq \inprod{F(x) + \eta_k H(x)}{x_{k+1}-x} + \inprod{\delta_{k+1}^F + \eta_k \delta_{k+1}^H}{x_{k+1}-x} +  \eta_k\mu_H\gnorm{x-x_{k+1}}{}{2},
         \end{aligned}}
         \end{equation}
         \begin{equation}\label{eq:recur-last-error}
             \begin{aligned}
                 \theta_k \inprod{\Delta \Ofrak_k}{x_{k+1}-x_k} = \theta_k \inprod{\Delta O_k}{x_{k+1}-x_k} + \theta_k \inprod{\delta_k^F -\delta_{k-1}^F+ \eta_{k-1}[\delta_k^H -\delta_{k-1}^H]}{x_{k+1}-x_k},
             \end{aligned}
         \end{equation}
         where $\Delta O_k = F(x_k) -F(x_{k-1}) + \eta_{k-1}[H(x_k) - H(x_{k-1})]$. Therefore, from \eqref{eq:recur-stoch}, \eqref{eq:delta-last-error} and \eqref{eq:recur-last-error}, we have 
         \begin{equation}\label{eq:recur-stoch-2}
             \begin{aligned}
                  &\inprod{F(x) + \eta_k H(x)}{x_{k+1}-x}\leq \tfrac{1}{2\gamma_k}[\gnorm{x-x_{k}}{}{2}-\gnorm{x_{k+1}-x_{k}}{}{2}]
                  \\&- (\tfrac{1}{2\gamma_k}  + \eta_k\mu_H)\gnorm{x-x_{k+1}}{}{2}\\
     &+ \inprod{\Delta \Ofrak_{k+1}}{x_{k+1}-x} - \theta_k\inprod{\Delta \Ofrak_k}{x_{k}-x}-\inprod{\delta_{k+1}^F + \eta_k \delta_{k+1}^H}{x_{k+1}-x}\\
       &-\theta_k \inprod{\Delta O_k}{x_{k+1}-x_k} - \theta_k \inprod{\delta_k^F -\delta_{k-1}^F+ \eta_{k-1}[\delta_k^H -\delta_{k-1}^H]}{x_{k+1}-x_k}.
             \end{aligned}
         \end{equation}
     Consequently, one can derive the upper bound for -$\theta_k \inprod{\Delta O_k}{x_{k+1}-x_k}$ from \eqref{eq:F-Lipschitz-property} and \eqref{eq:H-Lipschitz-property} and use \eqref{eq:recur-stoch-2}, to obtain \eqref{eq:one-iteration-iropex}. 
     \end{proof}
     \subsection{Convergence analysis for the general BVI}
     This section is dedicated to the detailed illustration of the proofs and materials we used to obtain the results of $\iropex$ for the general problem. We divide this section into two cases depending on whether we assume strong monotonicity.
     \subsubsection{Convergence analysis for the monotone problem}\label{sec:analysis-monotone}
This section provides the necessary analysis for the proof of Theorems \ref{thm:convex-case-optim-feasib}, \ref{thm:lower-bound-mono}. Moreover, we provide a detailed illustration analysis regarding Remark \ref{rem:alpha-H}. The following lemma states the necessary conditions to obtain the convergence results. First, let us mention the following useful propositions.
    \begin{proposition}\label{prop:tech_res1-mo}
	Let $\rho_1, \dots, \rho_j$ be a sequence of elements in $\Rbb^n$ and let $S$ be a convex set in $\Rbb^n$. Define the sequence $v_t, t = 1,\dots$, as follows: $v_1 \in S$ and 
	$%\begin{equation*}
	v_{t+1} = \argmin_{x \in S} \inprod{\rho_t}{x} + \tfrac{1}{2}\gnorm{x-v_t}{}{2}.
	$ %\end{equation*}
	Then, for any $x \in S$ and $t \ge 0$, the following inequality holds
	\begin{equation}\label{eq:int_rel114_1-mo}
		\inprod{\rho_t}{v_t-x} \le \tfrac{1}{2}\gnorm{x-v_t}{}{2}-\tfrac{1}{2} \gnorm{x-v_{t+1}}{}{2}+ \tfrac{1}{2}\gnorm{\rho_t}{}{2},
	\end{equation}
	% \begin{equation}\label{eq:int_rel114}
	% \tsum_{ t =0}^{j} \inprod{\rho_t}{v_t-x} \le \tfrac{1}{2}\gnorm{x-v_0}{2}{2} + \tfrac{1}{2}\tsum_{t=0}^{j} \gnorm{\rho_t}{2}{2}.
	% \end{equation}
\end{proposition}
\begin{proof}
	Using Lemma \ref{lem:three-point} with $g(x) =  \inprod{\rho_t}{x}$ 
	%, $W(y,x) = \tfrac{1}{2}\gnorm{y-x}{}{2}$ %, $\wt{x} = v_t$ 
, we have, due to the optimality of $v_{t+1}$,
	\begin{equation*}
		\inprod{\rho_t}{v_{t+1}-x} + \tfrac{1}{2}\gnorm{v_{t+1}-v_t}{2}{2} + \tfrac{1}{2} \gnorm{x-v_{t+1}}{}{2}\le  \tfrac{1}{2}\gnorm{x-v_t}{}{2},
	\end{equation*}
	is satisfied for all $x \in S$. The above relation and the fact
	$
	\inprod{\rho_t}{v_t-v_{t+1}} -\tfrac{1}{2}\gnorm{v_{t+1}-v_t}{}{2} \le \tfrac{1}{2}\gnorm{\rho_t}{}{2}$,
	imply \eqref{eq:int_rel114_1-mo}. 
	%    that 
	% \begin{equation*}
	% \inprod{\rho_t}{v_t-x} \le \tfrac{1}{2}\gnorm{x-v_t}{}{2}-\tfrac{1}{2} \gnorm{x-v_{t+1}}{}{2}+ \tfrac{1}{2}\gnorm{\rho_t}{}{2},
	% \end{equation*}
	% for all $x \in S$. Summing up the above relations from $t = 0$ to $j$ and noting the nonnegativity of $\gnorm{\cdot}{2}{2}$, we obtain \eqref{eq:int_rel114}. 
	Hence, we conclude the proof.
\end{proof}
\begin{lemma}\label{lemma:aug-seq}
    Consider the sequences $\{x^a_k\}_{k\geq 1}$, as follows
	\begin{equation}\label{eq:x-aug-seq_def-mo}
		\begin{aligned}
			&x_2^a =x_1^a =x_1, 
			\quad &&	x_{k+1}^a := \argmin_{x \in X_F^*} -\inprod{\gamma_{k-1}\delta_k}{x} + \tfrac{1}{2}\gnorm{x-x_k^a}{}{2},
			\quad&&\forall k\geq 2.
            \end{aligned}
        \end{equation}
      For $ \delta_ k  = \delta_{k}^F + \eta_k\delta_{k+1}^H$.  Suppose $ \tfrac{\tau_k}{\gamma_k\eta_k}\leq \tfrac{\tau_{k-1}}{\gamma_{k-1}\eta_{k-1}},$ for $k\geq 2$. Then, we have 
        \begin{equation}\label{eq:stoch_innerprod_delta-mo}
            \begin{aligned}
			&\forall x\in X_F^*,
			&&\tsum_{k=1}^{K-1} \tfrac{\tau_k}{\eta_k}\inprod{\delta_{k+1}}{x-{x}_{k+1}^a} 
			&&\le \tfrac{\tau_1}{2\gamma_1\eta_1}\gnorm{x- x_1}{}{2} 
			+ \tsum_{k=1}^{K-1} \tfrac{\gamma_k\tau_k}{2\eta_k}\gnorm{\delta_{k+1}}{}{2},
            \end{aligned}
        \end{equation}
\end{lemma}
\begin{proof}
    Noting the definition of $ \{x^a_k\}_{k\geq 2}$ in \eqref{eq:x-aug-seq_def-mo} and applying Proposition \ref{prop:tech_res1-mo} we have 
    \begin{equation}\label{eq:recur-augmented}
         - \inprod{\gamma_{k}\delta_{k+1}}{x_{k+1}^a -x}\leq \tfrac{1}{2}\gnorm{x-x_{k+1}^a}{}{2} - \tfrac{1}{2}\gnorm{x-x_{k+2}^a}{}{2} + \tfrac{\gamma_{k}^2}{2}\gnorm{\delta_{k+1}}{}{2},
    \end{equation} 
    Multiplying the above relation by $\tfrac{\tau_k}{\eta_k}$, summing it over $k=1$ to $K-1$ given that $\tfrac{\tau_k}{\gamma_k\eta_k}\leq \tfrac{\tau_{k-1}}{\gamma_{k-1}\eta_{k-1}}$, we get \eqref{eq:stoch_innerprod_delta-mo}. 
\end{proof}
\paragraph{Analysis of Theorem \ref{thm:convex-case-optim-feasib} }
\begin{lemma}\label{lemma:optim-feasib}
    Let us consider the step-size policy that satisfies the following conditions
    \begin{subequations}
        \begin{equation}\label{eq:cond-1-conv}
            \tfrac{\tau_k}{\gamma_k\eta_k}\leq \tfrac{\tau_{k-1}}{\gamma_{k-1}\eta_{k-1}},\quad \tfrac{\tau_k\theta_k}{\eta_k} =  \tfrac{\tau_{k-1}}{\eta_{k-1}},\quad  \theta_k (L_F^2 + \eta_k^2L_H^2)\leq \tfrac{1}{50\gamma_k\gamma_{k-1}}
        \end{equation}
        \begin{equation}\label{eq:cond-2-conv}
            \tfrac{\tau_k}{\gamma_k}\leq \tfrac{\tau_{k-1}}{\gamma_{k-1}},\quad \tau_k\theta_k =  \tau_{k-1},
        \end{equation}
    \end{subequations}
    Then, we have the optimality and feasibility bounds mentioned in Theorem \ref{thm:convex-case-optim-feasib}. 
\end{lemma}
\begin{proof}

First, note that \eqref{eq:step-size-con-bivel} satisfies \eqref{eq:cond-1-conv} and \eqref{eq:cond-2-conv}. Regarding Lemma \ref{lem:one-iteration-iropex} and taking $x=x^*_F\in X^*_F$ as an arbitrary solution to the inner problem $\text{VI}(F,X)$, we know that $\inprod{F(x^*_F)}{x_{k+1}-x^*_F}\geq 0$. Thus, noting $\eta_k\geq 0$, one can rewrite \eqref{eq:one-iteration-iropex} as follows 
\begin{equation}\label{eq:mono-optim-feasib-1}
\Scale[0.9]{
    \begin{aligned}
        &\inprod{H(x_F^*)}{x_{k+1}-x^*_F}\leq \tfrac{1}{2\gamma_k\eta_k}[\gnorm{x^*_F-x_{k}}{}{2}-\gnorm{x^*_F-x_{k+1}}{}{2} -  \gnorm{x_{k+1}-x_{k}}{}{2}]\\
        &+ \tfrac{1}{\eta_k}\inprod{\Delta \Ofrak_{k+1}}{x_{k+1}-x^*_F} - \tfrac{\theta_k}{\eta_k}\inprod{\Delta \Ofrak_k}{x_k-x^*_F}
         + \tfrac{\theta_k}{\eta_k}(L_F + \eta_{k-1} L_H)\gnorm{x_k-x_{k-1}}{}{}\gnorm{x_{k+1}-x_{k}}{}{} \\
         &+ \tfrac{\theta_k}{\eta_k}(M_F + \eta_{k-1}M_H)\gnorm{x_{k+1}-x_k}{}{}.\\
         & -\tfrac{1}{\eta_k}\inprod{\delta_{k+1}^F + \eta_k \delta_{k+1}^H}{x_{k+1}-x^*_F} - \tfrac{\theta_k}{\eta_k}\inprod{\delta_k^F -\delta_{k-1}^F+ \eta_{k-1}[\delta_k^H -\delta_{k-1}^H]}{x_{k+1}-x_k}.
    \end{aligned}}
\end{equation}
Multiplying both sides by $\tau_k\geq 0$ and summing up \eqref{eq:mono-optim-feasib-1} from $k=1$ to $K-1$, we have the following 
\begin{equation}\label{eq:mono-optim-feasib-2}
\Scale[1]{
    \begin{aligned}
        &\tsum_{k=1}^{K-1}\inprod{H(x_F^*)}{\tau_k(x_{k+1}-x^*_F)}\leq \tsum_{k=1}^{K-1}\tfrac{\tau_k}{2\gamma_k\eta_k}[\gnorm{x^*_F-x_{k}}{}{2}-\gnorm{x^*_F-x_{k+1}}{}{2}]\\
        &+ \tsum_{k=1}^{K-1}\tfrac{\tau_k}{\eta_k}\inprod{\Delta \Ofrak_{k+1}}{x_{k+1}-x^*_F} - \tfrac{\tau_k\theta_k}{\eta_k}\inprod{\Delta \Ofrak_k}{x_k-x^*_F}\\
         &+ \tsum_{k=1}^{K-1}\big[\tfrac{\tau_k\theta_k}{\eta_k}(L_F + \eta_{k-1} L_H)\gnorm{x_k-x_{k-1}}{}{}\gnorm{x_{k+1}-x_{k}}{}{} - \tfrac{\tau_k}{10\gamma_k\eta_k} \gnorm{x_{k+1}-x_{k}}{}{2}\\
         &- \tfrac{\tau_{k-1}}{10\gamma_{k-1}\eta_{k-1}} \gnorm{x_{k}-x_{k-1}}{}{2}]+ \tsum_{k=1}^{K-1}\tfrac{\tau_k\theta_k}{\eta_k}(M_F + \eta_{k-1}M_H)\gnorm{x_{k+1}-x_k}{}{}-\tfrac{\tau_k}{10\gamma_k\eta_k} \gnorm{x_{k+1}-x_{k}}{}{2}\\
         &- \tsum_{k=1}^{K-1}\tfrac{\tau_k}{\eta_k}\inprod{\delta_{k+1}^F + \eta_k \delta_{k+1}^H}{x_{k+1}-x^*_F}\\
         & -  \tsum_{k=1}^{K-1}\tfrac{\tau_k\theta_k}{\eta_k}\inprod{\delta_k^F -\delta_{k-1}^F}{x_{k+1}-x_k}-\tfrac{\tau_k}{10\gamma_k\eta_k} \gnorm{x_{k+1}-x_{k}}{}{2}\\
         & - \tsum_{k=1}^{K-1}\tfrac{\tau_k\theta_k\eta_{k-1}}{\eta_k}\inprod{\delta_k^H -\delta_{k-1}^H}{x_{k+1}-x_k}-\tfrac{\tau_k}{10\gamma_k\eta_k} \gnorm{x_{k+1}-x_{k}}{}{2}
    \end{aligned}}
\end{equation}
Moreover from Young's inequality ($ab\leq \tfrac{\epsilon}{2}a^2 + \tfrac{1}{2\epsilon}b^2$),  we have the followings 
\begin{subequations}\label{eq:smooth-stoch}
    \begin{equation}\label{eq:non-smooth-youngs}
        \tfrac{\tau_k\theta_k}{\eta_k}(M_F + \eta_{k-1}M_H)\gnorm{x_{k+1}-x_k}{}{}-\tfrac{\tau_k}{10\gamma_k\eta_k} \gnorm{x_{k+1}-x_{k}}{}{2}\leq \tfrac{5\theta^2\gamma_k\tau_k}{\eta_k}(M_F^2+\eta_{k-1}^2M_H^2) 
    \end{equation}
    \begin{equation}\label{eq:stoch-F-youngs-last}
        \tfrac{\tau_k\theta_k}{\eta_k}\inprod{\delta_k^F -\delta_{k-1}^F}{x_{k+1}-x_k}-\tfrac{\tau_k}{10\gamma_k\eta_k} \gnorm{x_{k+1}-x_{k}}{}{2} \leq \tfrac{5\theta^2\gamma_k\tau_k\gnorm{\delta_k^F -\delta_{k-1}^F}{}{2}}{\eta_k}
    \end{equation}
    \begin{equation}\label{eq:stoch-F-youngs}
        \tfrac{\tau_k\theta_k\eta_{k-1}}{\eta_k}\inprod{\delta_k^H -\delta_{k-1}^H}{x_{k+1}-x_k}-\tfrac{\tau_k}{10\gamma_k\eta_k} \gnorm{x_{k+1}-x_{k}}{}{2} \leq \tfrac{5\theta^2\gamma_k\tau_k\eta_{k-1}^2\gnorm{\delta_k^H -\delta_{k-1}^H}{}{2}}{\eta_k}
    \end{equation}
\end{subequations}
Now, from the monotonicity of $H$ and the conditions \eqref{eq:cond-1-conv}-\eqref{eq:cond-2-conv}, \eqref{eq:mono-optim-feasib-2}, and \eqref{eq:smooth-stoch}, we have 
\begin{equation}\label{eq:mono-optim-feasib-3}
\Scale[.9]{
    \begin{aligned}
        &\inprod{H(x_F^*)}{\bar{x}_{K}-x^*_F}\leq \tfrac{1}{\tsum_{k=1}^{K-1}\tau_k}\Big[[\tfrac{\tau_1}{2\gamma_1\eta_1}\gnorm{x^*_F-x_1}{}{2}-\tfrac{\tau_{K-1}}{2\gamma_{K-1}\eta_{K-1}}\gnorm{x^*_F-x_K}{}{2}] + \tfrac{\tau_{K-1}}{\eta_{K-1}}\inprod{\Delta \Ofrak_{K}}{x_{K}-x^*_F}\\
        & + \tsum_{k=1}^{K-1} \tfrac{5\theta^2\gamma_k\tau_k}{\eta_k}(M_F^2+\gnorm{\delta_k^F -\delta_{k-1}^F}{}{2}+\eta_{k-1}^2(M_H^2+\gnorm{\delta_k^F -\delta_{k-1}^F}{}{2}))
        \\
        &- \tsum_{k=1}^{K-1}\tfrac{\tau_k}{\eta_k}\inprod{\delta_{k+1}^F + \eta_k \delta_{k+1}^H}{x_{k+1}-x^*_F}\Big].
    \end{aligned}}
\end{equation}
Note $- \tsum_{k=1}^{K-1}\tfrac{\tau_k}{\eta_k}\inprod{\delta_{k+1}^F + \eta_k \delta_{k+1}^H}{x_{k+1}-x^*_F}$ is written as below 
\begin{equation}\label{eq:expectation-1}
\Scale[.9]{
\begin{aligned}
    - \tsum_{k=1}^{K-1}\tfrac{\tau_k}{\eta_k}\inprod{\delta_{k+1}^F + \eta_k \delta_{k+1}^H}{x_{k+1}-x^*_F} =& \tsum_{k=1}^{K-1}\tfrac{\tau_k}{\eta_k}\inprod{\delta_{k+1}^F + \eta_k \delta_{k+1}^H}{x^*_F-x_{k+1}^a}\\
    &+ \tsum_{k=1}^{K-1}\tfrac{\tau_k}{\eta_k}\inprod{\delta_{k+1}^F + \eta_k \delta_{k+1}^H}{x_{k+1}^a-x_{k+1}},
\end{aligned}}
\end{equation}
where $x_{k+1}^a$ is a sequence defined in Lemma \ref{lemma:aug-seq}. 
Moreover, let us consider $\xi_{[k]} :=(\xi_1,\dots,\xi_k) $. Thus, from \eqref{eq:x-update} and \eqref{eq:x-aug-seq_def-mo}, one can see that $x_{k+1}$ and $x_{k+1}^a$  depend on $\xi_{[k]}$. Therefore, from \eqref{eq:stoch-oracle-F}, \eqref{eq:stoch-oracle-H}, and from tower property of expectations, we have the following relation
\begin{equation}\label{eq:tower-propert}
    \Ebb[\inprod{\delta_{k+1}^F + \eta_k \delta_{k+1}^H}{x_{k+1}^a-x_{k+1}} = \Ebb[\Ebb[\inprod{\delta_{k+1}^F + \eta_k \delta_{k+1}^H}{x_{k+1}^a-x_{k+1}}|\xi_{[k]} ]] = 0.
\end{equation}
Furthermore, from \eqref{eq:stoch_innerprod_delta-mo}, 
 we can have the following upper bound for the first summation in the right-hand side of \eqref{eq:expectation-1}. 
\begin{equation}\label{eq:expectation-2}
\Scale[.9]{
\begin{aligned}
    \Ebb[\tsum_{k=1}^{K-1}\tfrac{\tau_k}{\eta_k}\inprod{\delta_{k+1}^F + \eta_k \delta_{k+1}^H}{x^*_F-x_{k+1}^a}] &\leq \Ebb[\tfrac{\tau_1}{2\gamma_1\eta_1}\gnorm{x^*_F - x_1}{}{2} ]+ \Ebb[ \tsum_{k=1}^{K-1} \tfrac{\gamma_k\tau_k}{2\eta_k}\gnorm{\delta_{k+1}^F + \eta_k \delta_{k+1}^H}{}{2}]\\
    & \leq \tfrac{\tau_1}{\gamma_1\eta_1}D_X^2 + \tsum_{k=1}^{K-1} \tfrac{\gamma_k\tau_k(\sigma_F^2 + \eta_k^2\sigma_H^2)}{\eta_k},
\end{aligned}}
\end{equation}
where the last inequality comes from $\gnorm{a+b}{}{2}\leq 2\gnorm{a}{}{2} + 2\gnorm{b}{}{2}$, \eqref{eq:stoch-oracle-F} and \eqref{eq:stoch-oracle-H}. Therefore by taking the expectations from both sides of \eqref{eq:mono-optim-feasib-3} and in view of \eqref{eq:expectation-1}-\eqref{eq:expectation-2}, $\gnorm{a+b}{}{2}\leq 2\gnorm{a}{}{2} + 2\gnorm{b}{}{2}$, Young's inequality for $\inprod{\Delta \Ofrak_{K}}{x_{K}-x^*_F}$ , $\tau_k =1$ and the third condition of \eqref{eq:cond-1-conv} in \eqref{eq:mono-optim-feasib-3}, we obtain the following 
\begin{equation}\label{eq:mono-optim-feasib-3-1}
    \begin{aligned}
        \Ebb[\inprod{H(x_F^*)}{\bar{x}_{K}-x^*_F}]\leq & \tfrac{2}{K\gamma_1\eta_1}D_X^2 + \tfrac{5\gamma_{K-1}}{K\eta_{K-1}}(M_F^2+2\sigma_F^2+\eta_{K-1}^2(M_H^2+2\sigma_H^2))\\
        & + \tsum_{k=1}^{K-1} \tfrac{5\theta^2\gamma_k}{K\eta_k}(M_F^2+2\sigma_F^2+\eta_{k-1}^2(M_H^2+2\sigma_H^2))+\tsum_{k=1}^{K-1} \tfrac{\gamma_k(\sigma_F^2 + \eta_k^2\sigma_H^2)}{K\eta_k}.
    \end{aligned}
\end{equation}
Using the step-size policy in \eqref{eq:step-size-con-bivel}, one can derive the following bound for \eqref{eq:mono-optim-feasib-3-1}. 
\begin{equation}\label{eq:mono-optim-feasib-4}
\Scale[1]{
\begin{aligned}
        \Ebb[\inprod{H(x_F^*)}{\bar{x}_{K}-x^*_F}]\leq & D_X\bigg[16D_X(\tfrac{L_F}{K^{3/4}} +\tfrac{L_H}{K}) + \tfrac{2\sqrt{M_F^2+2\sigma_F^2+\eta^2(M_H^2+ 2\sigma_H^2)}}{K^{1/4}} \\ &+\tfrac{5(M_F^2+2\sigma_F^2+\eta^2(M_H^2+2\sigma_H^2)) + \sigma_F^2 + \eta^2\sigma_H^2}{8K^{-1/4}(L_F + \eta L_H) + K^{1/4}\sqrt{M_F^2+2\sigma_F^2+\eta^2(M_H^2+2\sigma_H^2)}}\\
         &+ \tfrac{5(M_F^2+2\sigma_F^2+\eta^2(M_H^2+2\sigma_H^2))}{8K^{3/4}(L_F + \eta L_H) + K^{5/4}\sqrt{M_F^2+2\sigma_F^2+\eta^2(M_H^2+\sigma_H^2)}}\bigg].
    \end{aligned}}
\end{equation}
Therefore, by taking the maximum on both sides of \eqref{eq:mono-optim-feasib-4} concerning the set $X^*_F$ and using the definition of the optimality gap function in \eqref{eq:optim-gap}, one obtains the following 
\begin{equation}\label{eq:mono-optim-feasib-5}
\Scale[1]{
\begin{aligned}
    \Ebb[\gap(\bar{x}_K, H, X^*_F)] \leq  & D_X\bigg[16D_X(\tfrac{L_F}{K^{3/4}} +\tfrac{L_H}{K}) + \tfrac{2\sqrt{M_F^2+2\sigma_F^2+\eta^2(M_H^2+ 2\sigma_H^2)}}{K^{1/4}} \\&+\tfrac{5(M_F^2+2\sigma_F^2+\eta^2(M_H^2+2\sigma_H^2)) + \sigma_F^2 + \eta^2\sigma_H^2}{8K^{-1/4}(L_F + \eta L_H) + K^{1/4}\sqrt{M_F^2+2\sigma_F^2+\eta^2(M_H^2+2\sigma_H^2)}}\\
         &+ \tfrac{5(M_F^2+2\sigma_F^2+\eta^2(M_H^2+2\sigma_H^2))}{8K^{3/4}(L_F + \eta L_H) + K^{5/4}\sqrt{M_F^2+2\sigma_F^2+\eta^2(M_H^2+\sigma_H^2)}}\bigg].
     \end{aligned}}
\end{equation}
Now, let us move to the feasibility gap function at the inner level. Considering Lemma \ref{lem:one-iteration-iropex} and from Cauchy-Schwarz inequality, we obtain the following 
\begin{equation}\label{eq:feasib-1}
 \Scale[.95]{
     \begin{aligned}
         \inprod{F(x)}{x_{k+1}-x}\leq &\tfrac{1}{2\gamma_k}[\gnorm{x-x_{k}}{}{2}-\gnorm{x-x_{k+1}}{}{2} -  \gnorm{x_{k+1}-x_{k}}{}{2}] + 2C_HD_X\eta_k\\
        &+ \inprod{\Delta \Ofrak_{k+1}}{x_{k+1}-x} - \theta_k\inprod{\Delta \Ofrak_k}{x_k-x}\\
         &+ \theta_k(L_F + \eta_{k-1} L_H)\gnorm{x_k-x_{k-1}}{}{}\gnorm{x_{k+1}-x_{k}}{}{} \\
         &+ \theta_k(M_F + \eta_{k-1}M_H)\gnorm{x_{k+1}-x_k}{}{}\\
         & -\inprod{\delta_{k+1}^F + \eta_k \delta_{k+1}^H}{x_{k+1}-x} -\theta_k\inprod{\delta_k^F -\delta_{k-1}^F+ \eta_{k-1}[\delta_k^H -\delta_{k-1}^H]}{x_{k+1}-x_k}.
     \end{aligned}}
     \end{equation}
 Using a similar approach to obtain an optimality gap and using the conditions in \eqref{eq:cond-1-conv}-\eqref{eq:cond-2-conv}, we obtain the following 
 \begin{equation}\label{eq:mono-optim-feasib-6}
    \begin{aligned}
        \Ebb[\inprod{F(x)}{\bar{x}_K-x}]\leq & \tfrac{2}{K\gamma_1}D_X^2 + \tfrac{5\gamma_{K-1}}{K}(M_F^2+2\sigma_F^2+\eta_{K-1}^2(M_H^2+\sigma_H^2))+\tfrac{2 C_H D_X\tsum_{k=1}^{K-1}\eta_k}{K}\\
        & + \tsum_{k=1}^{K-1} \tfrac{5\theta^2\gamma_k}{K}(M_F^2+2\sigma_F^2+\eta_k^2(M_H+2\sigma_H^2))+\tsum_{k=1}^{K-1} \tfrac{\gamma_k(\sigma_F^2 + \eta_k^2\sigma_H^2)}{K}.
    \end{aligned}
\end{equation}
From the step-size policy in \eqref{eq:step-size-con-bivel}, one can drive the following 
\begin{equation}\label{eq:mono-optim-feasib-7}
\Scale[1]{
    \begin{aligned}
        \Ebb[\inprod{F(x)}{\bar{x}_K-x}]\leq &   D_X\bigg[16D_X(\tfrac{L_F}{K} +\tfrac{L_H}{K^{5/4}}) + \tfrac{2\sqrt{M_F^2+2\sigma_F^2+\eta^2(M_H^2+2\sigma_H^2)}}{K^{1/2}} \\&+\tfrac{5(M_F^2+2\sigma_F^2+\eta^2(M_H^2+2\sigma_H^2))+ \sigma_F^2 + \eta^2\sigma_H^2}{8(L_F + \eta L_H) + K^{1/2}\sqrt{M_F^2+2\sigma_F^2+\eta^2(M_H^2+2\sigma_H^2)}}
         \\
         &+ \tfrac{5(M_F^2+2\sigma_F^2+\eta^2(M_H^2+\sigma_H^2))}{8K(L_F + \eta L_H) + K^{3/2}\sqrt{M_F^2+2\sigma_F^2+\eta^2(M_H^2+2\sigma_H^2)}}+\tfrac{2 C_H }{K^{1/4}}\bigg].
    \end{aligned}}
\end{equation}
Taking the maximum from both sides of \eqref{eq:mono-optim-feasib-7} concerning set $X$, we obtain the following feasibility bound  
\begin{equation}\label{eq:mono-optim-feasib-8}
\Scale[1]{
    \begin{aligned}
        \Ebb[\gap(\bar{x}_K,F,X)]\leq &  D_X\bigg[16D_X(\tfrac{L_F}{K} +\tfrac{L_H}{K^{5/4}}) + \tfrac{2\sqrt{M_F^2+2\sigma_F^2+\eta^2(M_H^2+2\sigma_H^2)}}{K^{1/2}} \\&+\tfrac{5(M_F^2+2\sigma_F^2+\eta^2(M_H^2+2\sigma_H^2))+ \sigma_F^2 + \eta^2\sigma_H^2}{8(L_F + \eta L_H) + K^{1/2}\sqrt{M_F^2+2\sigma_F^2+\eta^2(M_H^2+2\sigma_H^2)}}
         \\
         &+ \tfrac{5(M_F^2+2\sigma_F^2+\eta^2(M_H^2+\sigma_H^2))}{8K(L_F + \eta L_H) + K^{3/2}\sqrt{M_F^2+2\sigma_F^2+\eta^2(M_H^2+2\sigma_H^2)}}+\tfrac{2 C_H }{K^{1/4}}\bigg].
    \end{aligned}}
\end{equation}
For the lower bound in \eqref{eq:mono-optim}, note that $\bar{x}_K\in X$. From Definition \ref{def:optim-gap}, 
 and for any $\Tilde{x} \in X_F^*$, one can write 
\[
\inprod{H(\Tilde{x})}{\bar{x}_K - \Tilde{x}}\leq \max_{x\in X^*_F} \inprod{H(x)}{\bar{x}_K-x}  = \gap(\bar{x}_K, H, X^*_F)
\]
From Cauchy-Schwarz inequality and Assumption \eqref{eq:boundedness-H}, we have 
\[
-B_H \gnorm{\bar{x}_K - \Tilde{x}}{}{}\leq \gap(\bar{x}_K, H, X^*_F), 
\]
Letting $\Tilde{x} = \proj_{X^*_F}(\bar{x}_K)$ and taking the expectation from both sides, we obtain the results.
\end{proof}
\vspace{-3mm}
\paragraph{Analysis of Theorem \ref{thm:lower-bound-mono} }
\vspace{-3mm}
\begin{proof}
    Let us take $x^*_F \in X_F^*$ as an arbitrary vector. Then from Definition \ref{def:weak-sharp}, one can obtain
    \[
    \Ebb[\inprod{F(x_F^*)}{\bar{x}_K - x_F^*}]\geq \alpha\Ebb[\text{dist}(\bar{x}_K,X^*_F)].
    \] 
    Therefore, from the definition of the feasibility gap function in \eqref{eq:feasib-gap}, we have the following     \begin{equation}\label{eq:feasib-gap-lower}        \gap(\bar{x}_K,F,X)=\max_{x\in X}\Ebb[\inprod{F(x)}{\bar{x}_K-x}] \geq  \Ebb[\inprod{F(x^*_F)}{\bar{x}_K-x^*_F}]\geq \alpha \Ebb[\text{dist}(\bar{x}_K,X^*_F)]. 
    \end{equation}
    \vspace{-3mm}
    From \eqref{eq:mono-optim}, \eqref{eq:mono-feasib} and \eqref{eq:feasib-gap-lower}, we obtain \eqref{eq:lower-bound-mono}.  
\end{proof}
\paragraph{Analysis of Remark \ref{rem:alpha-H}}
The following Theorem formally states the rates concerning Remark \ref{rem:alpha-H}. 
\begin{theorem}\label{thm:weakly-sharp-conv}
         Suppose $\text{VI}(F,X)$ is $\alpha$-weakly sharp, and we have the following step-size policy 
         \begin{equation}\label{eq:step-size-con-bivel-weak-sharp}
         \Scale[.8]{
             \tau_k =\tau =1,\quad \theta_k =\theta=1,\quad \eta_k= \eta\leq  \tfrac{\alpha}{2\gnorm{H(x^*)}{}{}},\quad \gamma_k = \gamma =\tfrac{D_X}{8D_X(L_F + \eta L_H)+ \sqrt{K(M_F^2 +2\sigma_F^2+\eta^2 (M_H^2+2\sigma_H^2))}},}
         \end{equation}
         then, we have the following 
         \vspace{-2mm}
         \begin{equation*}\label{eq:weak-optim}
             \Scale[1]{
     \begin{aligned}
   \Ebb[ \gap(\bar{x}_K, H, X^*_F)] \leq &\tfrac{\gnorm{H(x^*)}{}{}D_X}{\alpha}\bigg[16D_X(\tfrac{L_F}{K} +\tfrac{L_H}{K^{5/4}}) + \tfrac{2\sqrt{M_F^2+2\sigma_F^2+\eta^2(M_H^2+2\sigma_H^2)}}{K^{1/2}} \\&+\tfrac{2[5(M_F^2+2\sigma_F^2+\eta^2(M_H^2+2\sigma_H^2))+ \sigma_F^2 + \eta^2\sigma^2_H]}{8D_X(L_F + \eta L_H) + K^{1/2}\sqrt{M_F^2+2\sigma_F^2+\eta^2(M_H^2+2\sigma_H^2)}}\\
       &  + \tfrac{10(M_F^2+2\sigma_F^2+\eta^2(M_H^2+2\sigma_H^2))}{8 KD_X(L_F + \eta L_H) + K^{3/2}\sqrt{M_F^2+2\sigma_F^2+\eta^2(M_H^2+2\sigma_H^2)}}\bigg].
         \end{aligned}}
         \end{equation*}
         and 
         \begin{equation*}\label{eq:weak-feasib}
   \Scale[1]{
   \begin{aligned}
   \Ebb[\text{dist}(\bar{x}_K,X^*_F)] \leq& \tfrac{D_X}{\alpha}\bigg[32D_X(\tfrac{L_F}{K} +\tfrac{L_H}{K^{5/4}}) + \tfrac{4\sqrt{M_F^2+2\sigma_F^2+\eta^2(M_H^2+2\sigma_H^2)}}{K^{1/2}} \\&+\tfrac{10(M_F^2+2\sigma_F^2+\eta^2(M_H^2+2\sigma_H^2))+ 2(\sigma_F^2 + \eta^2\sigma^2_H)}{8D_X(L_F + \eta L_H) + K^{1/2}\sqrt{M_F^2+2\sigma_F^2+\eta^2(M_H^2+2\sigma_H^2)}}\\
       &  + \tfrac{10(M_F^2+2\sigma_F^2+\eta^2(M_H^2+2\sigma_H^2)) }{8 KD_X(L_F + \eta L_H) + K^{3/2}\sqrt{M_F^2+2\sigma_F^2+\eta^2(M_H^2+2\sigma_H^2)}}\bigg].
\end{aligned}.}
         \end{equation*}
     \end{theorem}
\begin{proof}
%\vspace{-5mm}
Take an arbitrary solution $x^*\in X^*_H$ where $X^*_H$ denotes the solution set of problem \eqref{eq:bilevel-vi}. Using Lemma \ref{lem:one-iteration-iropex} and the similar relations we used in Lemma \ref{lemma:optim-feasib}, we have the following 
\begin{equation}\label{eq:weak-1}
\Scale[.9]{
    \begin{aligned}
       \Ebb[ \inprod{F(x^*)+ \eta H(x^*)}{\bar{x}_K-x^*}]\leq &\tfrac{2}{K\gamma_1}D_X^2 + \tfrac{5\gamma_{K-1}\theta_{K-1}^2}{K}(M_F^2+2\sigma_F^2+\eta_{K-1}^2(M_H^2+2\sigma_H^2))\\
        & + \tsum_{k=1}^{K-1} \tfrac{5\theta^2\gamma_k}{K}(M_F^2+2\sigma_F^2+\eta_{k-1}^2(M_H^2+\sigma_H^2))+\tsum_{k=1}^{K-1} \tfrac{\gamma_k(\sigma_F^2 + \eta_k^2\sigma_H^2)}{K}.
    \end{aligned}}
\end{equation}
Given that $x^*\in X^*_F$ and from Definition \ref{def:weak-sharp}, we have 
\begin{equation*}
    \Ebb[\inprod{F(x^*)}{\bar{x}_K-x^*}]\geq \alpha \Ebb[\text{dist}(\bar{x}_K,X *_F)].
\end{equation*}
Then 
\begin{equation}
\begin{aligned}
     \Ebb[\inprod{H(x^*)}{\bar{x}_K-x^*}] =&  \Ebb[\inprod{H(x^*)}{\bar{x}_K-\proj_{X^*_F}(\bar{x}_K)+\proj_{X^*_F}(\bar{x}_K)-x^*}]\\
     &\geq -\gnorm{H(x^*)}{}{}\Ebb[\gnorm{\bar{x}_K-\proj_{X^*_F}(\bar{x}_K)}{}{}].
\end{aligned}
\end{equation}
Note that given $\proj_{X^*_F}(\bar{x}_K) \in X^*_F$ and $x^*$ solves VI$(H,X^*_F)$, we know $\Ebb[\inprod{H(x^*)}{\proj_{X^*_F}(\bar{x}_K)-x^*}]\geq 0$. Therefore from \eqref{eq:weak-1}, we have 
\begin{equation}
\Scale[.9]{
\begin{aligned}
        \Ebb[\alpha& \text{dist}(\bar{x}_K,X^*_F) - \eta \gnorm{H(x^*)}{}{}\text{dist}(\bar{x}_K,X^*_F)] \leq  \tfrac{2D_X^2}{K\gamma_1} + \tfrac{5\gamma_{K-1}\theta_{K-1}^2}{K}(M_F^2+2\sigma_F^2+\eta_{K-1}^2(M_H^2+2\sigma_H^2))\\
         &+ \tsum_{k=1}^{K-1} \tfrac{5\theta^2\gamma_k}{K}(M_F^2+2\sigma_F^2+\eta_k^2(M_H^2+2\sigma_H^2))+\tsum_{k=1}^{K-1} \tfrac{\gamma_k(\sigma_F^2 + \eta_k^2\sigma_H^2)}{K}.
\end{aligned}}
\end{equation}
Taking step-size policy in \eqref{eq:step-size-con-bivel-weak-sharp}, we have 
\begin{equation*}
\Scale[1]{
\begin{aligned}
    \Ebb[\text{dist}(\bar{x}_K,X^*_F)] \leq& \tfrac{D_X}{\alpha}\bigg[32D_X(\tfrac{L_F}{K} +\tfrac{L_H}{K^{5/4}}) + \tfrac{4\sqrt{M_F^2+2\sigma_F^2+\eta^2(M_H^2+2\sigma_H^2)}}{K^{1/2}} \\&+\tfrac{10(M_F^2+2\sigma_F^2+\eta^2(M_H^2+2\sigma_H^2))+ 2(\sigma_F^2 + \eta^2\sigma^2_H)}{8D_X(L_F + \eta L_H) + K^{1/2}\sqrt{M_F^2+2\sigma_F^2+\eta^2(M_H^2+2\sigma_H^2)}}\\
       &  + \tfrac{10(M_F^2+2\sigma_F^2+\eta^2(M_H^2+2\sigma_H^2)) }{8 KD_X(L_F + \eta L_H) + K^{3/2}\sqrt{M_F^2+2\sigma_F^2+\eta^2(M_H^2+2\sigma_H^2)}}\bigg]
\end{aligned}
    }
\end{equation*}
Moreover, from \eqref{eq:mono-optim-feasib-3-1} and applying the step-size policy in \eqref{eq:step-size-con-bivel-weak-sharp}, one can get the following 
\begin{equation}\label{eq:mono-optim-feasib-weak-2}
\Scale[1]{
    \begin{aligned}
        \Ebb[\inprod{H(x_F^*)}{\bar{x}_{K}-x^*_F}]\leq & \tfrac{\gnorm{H(x^*)}{}{}D_X}{\alpha}\bigg[16D_X(\tfrac{L_F}{K} +\tfrac{L_H}{K^{5/4}}) + \tfrac{2\sqrt{M_F^2+2\sigma_F^2+\eta^2(M_H^2+2\sigma_H^2)}}{K^{1/2}} \\&+\tfrac{2[5(M_F^2+2\sigma_F^2+\eta^2(M_H^2+2\sigma_H^2))+ \sigma_F^2 + \eta^2\sigma^2_H]}{8D_X(L_F + \eta L_H) + K^{1/2}\sqrt{M_F^2+2\sigma_F^2+\eta^2(M_H^2+2\sigma_H^2)}}\\
       &  + \tfrac{10(M_F^2+2\sigma_F^2+\eta^2(M_H^2+2\sigma_H^2))}{8 KD_X(L_F + \eta L_H) + K^{3/2}\sqrt{M_F^2+2\sigma_F^2+\eta^2(M_H^2+2\sigma_H^2)}}\bigg].
    \end{aligned}}
\end{equation}
Taking the maximum from both sides, we obtain 
\begin{equation*}
\Scale[1]{
\begin{aligned}
    \Ebb[ \gap(\bar{x}_K, H, X^*_F)] \leq &\tfrac{\gnorm{H(x^*)}{}{}D_X}{\alpha}\bigg[16D_X(\tfrac{L_F}{K} +\tfrac{L_H}{K^{5/4}}) + \tfrac{2\sqrt{M_F^2+2\sigma_F^2+\eta^2(M_H^2+2\sigma_H^2)}}{K^{1/2}} \\&+\tfrac{2[5(M_F^2+2\sigma_F^2+\eta^2(M_H^2+2\sigma_H^2))+ \sigma_F^2 + \eta^2\sigma^2_H]}{8D_X(L_F + \eta L_H) + K^{1/2}\sqrt{M_F^2+2\sigma_F^2+\eta^2(M_H^2+2\sigma_H^2)}}\\
       &  + \tfrac{10(M_F^2+2\sigma_F^2+\eta^2(M_H^2+2\sigma_H^2))}{8 KD_X(L_F + \eta L_H) + K^{3/2}\sqrt{M_F^2+2\sigma_F^2+\eta^2(M_H^2+2\sigma_H^2)}}\bigg].
         \end{aligned}}
\end{equation*}
\end{proof}
\subsubsection{Convergence results for strongly monotone problem}\label{sec:strong-mono}
\begin{theorem}\label{thm:convex-case-optim-feasib-strong}
         Let us assume problem \eqref{eq:bilevel-vi} is strongly monotone $(\mu_H > 0)$. Suppose we have the following step-size policy 
         \vspace{-3mm}
         \begin{equation}\label{eq:step-size-con-bivel-strong}
         \Scale[.95]{
             \tau_k =k+1,\quad \theta_k =\tfrac{k}{k+1},\quad \eta_k= \eta= K^{-\tfrac{1}{4}},\quad \gamma_k = \gamma =\tfrac{D_X}{8D_X(L_F + \eta L_H)+ \sqrt{K(M_F^2 +2\sigma_F^2+\eta^2 (M_H^2+2\sigma_H^2))}},}
         \end{equation}
         \vspace{-2mm}
         Then, for $K\geq \tfrac{1}{2\gamma_k \eta_{k-1}\mu_H}$, Algorithm \ref{alg:IRopex} gives the following optimality and feasibility gaps 
         \begin{equation}\label{eq:mono-optim-strong}
\Scale[.9]{
\begin{aligned}
   -B_H\Ebb[\text{dist}(\bar{x}_K,X^*_F)] \leq  \Ebb[\gap(\bar{x}_K, H, X^*_F)] \leq  &D_X\bigg[\big(16D_X(\tfrac{L_F}{K^{7/4}} +\tfrac{L_H}{K^2}) + \tfrac{2\sqrt{M_F^2+2\sigma_F^2+\eta^2(M_H^2+ 2\sigma_H^2)}}{K^{5/4}} \big)\\
   &+\tfrac{5(M_F^2+2\sigma_F^2+\eta^2(M_H^2+2\sigma_H^2)) + \sigma_F^2 + \eta^2\sigma_H^2}{8K^{3/4}D_X(L_F + \eta L_H) + K^{5/4}\sqrt{M_F^2+2\sigma_F^2+\eta^2(M_H^2+2\sigma_H^2)}}\\
         &+ \tfrac{5(M_F^2+2\sigma_F^2+\eta^2(M_H^2+2\sigma_H^2))}{8K^{7/4}D_X(L_F + \eta L_H) + K^{9/4}\sqrt{M_F^2+2\sigma_F^2+\eta^2(M_H^2+\sigma_H^2)}}\bigg],
     \end{aligned}}
\end{equation}
\vspace{-2mm}
and 
\begin{equation}\label{eq:mono-feasib-strong}
\Scale[.9]{
    \begin{aligned}
        &\Ebb[\gap(\bar{x}_K,F,X)]\leq   D_X\bigg[\big(16D_X(\tfrac{L_F}{K^2} +\tfrac{L_H}{K^{9/4}}) + \tfrac{2\sqrt{M_F^2+2\sigma_F^2+\eta^2(M_H^2+2\sigma_H^2)}}{K^{3/2}} \big)\\
        &+\tfrac{5(M_F^2+2\sigma_F^2+\eta^2(M_H^2+2\sigma_H^2)) + \sigma_F^2 + \eta^2\sigma_H^2}{8K(L_F + \eta L_H) + K^{3/2}\sqrt{M_F^2+2\sigma_F^2+\eta^2(M_H^2+2\sigma_H^2)}}
         + \tfrac{5(M_F^2+2\sigma_F^2+\eta^2(M_H^2+\sigma_H^2))}{8K^2(L_F + \eta L_H) + K^{5/2}\sqrt{M_F^2+2\sigma_F^2+\eta^2(M_H^2+2\sigma_H^2)}}+\tfrac{2 C_H}{K^{5/4}}\bigg].
    \end{aligned}}
\end{equation}
Further, the explicit lower bound for the strongly monotone case has the following form
\begin{equation}\label{eq:lower-bound-mono-strong}
           \Scale[1]{
           \begin{aligned}
                \Ebb[\gap(\bar{x}_K, H, X^*_F)]& \geq  -\tfrac{B_H}{\alpha}\Bigg[D_X\big[\big(16D_X(\tfrac{L_F}{K^2} +\tfrac{L_H}{K^{9/4}}) + \tfrac{2\sqrt{M_F^2+2\sigma_F^2+\eta^2(M_H^2+2\sigma_H^2)}}{K^{3/2}} \big)\\&+\tfrac{5(M_F^2+2\sigma_F^2+\eta^2(M_H^2+2\sigma_H^2)) + \sigma_F^2 + \eta^2\sigma_H^2}{8KD_X(L_F + \eta L_H) + K^{3/2}\sqrt{M_F^2+2\sigma_F^2+\eta^2(M_H^2+2\sigma_H^2)}}
         \\
         &+ \tfrac{5(M_F^2+2\sigma_F^2+\eta^2(M_H^2+\sigma_H^2))}{8K^2D_X(L_F + \eta L_H) + K^{5/2}\sqrt{M_F^2+2\sigma_F^2+\eta^2(M_H^2+2\sigma_H^2)}}\big]+\tfrac{2 C_H}{K^{5/4}}\Bigg].
           \end{aligned}}
           \end{equation}
Moreover, consider the following step-size policy 
\begin{equation}\label{eq:step-size-con-bivel-weak-sharp-str}
         \Scale[.8]{
             \tau_k =k+1,\quad \theta_k =\tfrac{k}{k+1},\quad \eta_k= \eta= \tfrac{\alpha}{2\gnorm{H(x^*)}{}{}},\quad \gamma_k = \gamma =\tfrac{D_X}{8D_X(L_F + \eta L_H)+ \sqrt{K(M_F^2 +2\sigma_F^2+\eta^2 (M_H^2+2\sigma_H^2))}},}
         \end{equation}
         Then, similar to results in Theorem \ref{thm:weakly-sharp-conv} and under the weak sharpness assumption for $F$, we obtain the following bounds for $K\geq \tfrac{1}{2\gamma_k \eta_{k-1}\mu_H}$
        
         \begin{equation}\label{eq:weak-optim-str}
             \Scale[1]{
     \begin{aligned}
    \Ebb[ \gap(\bar{x}_K, H, X^*_F)] \leq &\tfrac{\gnorm{H(x^*)}{}{}D_X^2}{\alpha}\bigg[32(\tfrac{L_F}{K^2} +\tfrac{L_H}{K^{9/4}}) + \tfrac{4\sqrt{M_F^2+2\sigma_F^2+\eta^2(M_H^2+2\sigma_H^2)}}{K^{3/2}} \\&+\tfrac{2[5(M_F^2+2\sigma_F^2+\eta^2(M_H^2+2\sigma_H^2)) + \sigma_F^2 + \eta^2\sigma_H^2]}{8K(L_F + \eta L_H) + K^{3/2}\sqrt{M_F^2+2\sigma_F^2+\eta^2(M_H^2+2\sigma_H^2)}}\\
       &  + \tfrac{10(M_F^2+2\sigma_F^2+\eta^2(M_H^2+2\sigma_H^2))}{8 K^2(L_F + \eta L_H) +  K^{5/2}\sqrt{M_F^2+2\sigma_F^2+\eta^2(M_H^2+2\sigma_H^2)}}\bigg],
         \end{aligned}}
         \end{equation}
          \begin{equation}\label{eq:weak-feasib-str}
   \Scale[1]{
   \begin{aligned}
    \Ebb[\text{dist}(\bar{x}_K,X^*_F)] \leq& \tfrac{D_X}{\alpha}\bigg[32D_X(\tfrac{L_F}{K^2} +\tfrac{L_H}{K^{9/4}}) + \tfrac{4\sqrt{M_F^2+2\sigma_F^2+\eta^2(M_H^2+2\sigma_H^2)}}{K^{3/2}} \\&+\tfrac{10(M_F^2+2\sigma_F^2+\eta^2(M_H^2+2\sigma_H^2)) + 2(\sigma_F^2 + \eta^2\sigma_H^2)}{8KD_X(L_F + \eta L_H) + K^{3/2}\sqrt{M_F^2+2\sigma_F^2+\eta^2(M_H^2+2\sigma_H^2)}}\\
       &  + \tfrac{10(M_F^2+2\sigma_F^2+\eta^2(M_H^2+2\sigma_H^2))}{8 K^2D_X(L_F + \eta L_H) + K^{5/2}\sqrt{M_F^2+2\sigma_F^2+\eta^2(M_H^2+2\sigma_H^2)}}\bigg],
\end{aligned}}
         \end{equation}
     \end{theorem} 
\paragraph{Convergence analysis for the strongly monotone problem}

     \begin{proof}
         Since the proof of Theorem \ref{thm:convex-case-optim-feasib-strong} is very similar to the proof we had for Theorems \ref{thm:convex-case-optim-feasib}-\ref{thm:weakly-sharp-conv}, we only mention the differences. To prove \eqref{eq:mono-optim-strong} and \eqref{eq:mono-feasib-strong}, we have to satisfy the changes we have in conditions \eqref{eq:cond-1-conv} and \eqref{eq:cond-2-conv}. In particular, the new conditions to satisfy are the following 
          \begin{subequations}\label{}
        \begin{equation}\label{eq:cond-1-conv-str}
            \tfrac{\tau_k}{2\gamma_k\eta_k}\leq \tfrac{\tau_{k-1}}{\eta_{k-1}}(\tfrac{1}{2\gamma_{k-1}} + \eta_{k-1}\mu_H),\quad \tfrac{\tau_k\theta_k}{\eta_k} =  \tfrac{\tau_{k-1}}{\eta_{k-1}},\quad  \theta_k (L_F^2 + \eta_k^2L_H^2)\leq \tfrac{1}{50\gamma_k\gamma_{k-1}}
        \end{equation}
        \begin{equation}\label{eq:cond-2-conv-str}
            \tfrac{\tau_k}{2\gamma_k}\leq \tau_{k-1}(\tfrac{1}{2\gamma_k}\eta_{k-1}\mu_H),\quad \tau_k\theta_k =  \tau_{k-1},
        \end{equation}
    \end{subequations}
    Considering the new conditions in \eqref{eq:cond-1-conv-str} and \eqref{eq:cond-2-conv-str}, and following the same procedure used in Lemma \ref{lemma:optim-feasib} and the analysis of Theorem \ref{thm:convex-case-optim-feasib}—specifically with respect to the step-size policy in \eqref{eq:step-size-con-bivel-strong}—we derive \eqref{eq:mono-optim-strong} and \eqref{eq:mono-feasib-strong}. The lower bound in \eqref{eq:lower-bound-mono-strong} follows from the same argument presented in Theorem \ref{thm:lower-bound-mono}. Finally, the upper bounds in \eqref{eq:weak-feasib-str} and \eqref{eq:weak-optim-str} are established using the approach taken in Theorem \ref{thm:weakly-sharp-conv}.
     \end{proof}
     \subsubsection{Convergence analysis of variable step-size policy}\label{sec:adaptive}
    In this section, we analyze the convergence rates of the update policy introduced in Section \ref{sec:adaptive} where $\eta_k$ does not depend on $K$. 
    \begin{theorem}\label{thm:adaptive}
        Consider the following conditions 
        \begin{equation}\label{eq:adap-con}
            \tfrac{\tau_k \theta_k}{\eta_k} = \tfrac{\tau_{k-1}}{\eta_{k-1}},\quad \theta_k(L_F^2 + \eta_k^2L_H^2)\leq \tfrac{1}{50\gamma_k\gamma_{k-1}}.
        \end{equation}
        Then the following step-size policy satisfies \eqref{eq:adap-con} 
        \begin{equation}\label{eq:adap-step-size}
        \Scale[.9]{
            \tau_k = \tau = 1, \quad \eta_k = (k+1)^{-1/4}, \quad \theta_k = (\tfrac{k}{k+1})^{1/4}, \quad \gamma_k = \tfrac{D_X}{8D_X (L_F + \eta_k L_H) + \sqrt{k(M_F^2 + 2\sigma_F^2 + \eta_k^2(M_H^2+2\sigma_H^2))}}},
        \end{equation}
        and gives the following upper bounds for optimality and feasibility gaps 
        \begin{equation}\label{eq:adapt-optim-mono}
\Scale[0.85]{
    \begin{aligned}
        -B_H&\Ebb[\text{dist}(\bar{x}_K,X_F^*)]\leq \Ebb[\gap(\bar{x}_K, H, X_F^*)]\leq\tfrac{D_X^2 +D_{X_F^*}^2}{D_X}\bigg[\tfrac{16D_X(L_F + \eta_{K-1}L_H)}{K^{3/4}} + \tfrac{2\sqrt{M_F^2 + 2\sigma_F^2 + \eta^2_{K-1}(M_H^2 + 2\sigma_H^2)}}{K^{1/4}}\bigg]\\
        & + \tfrac{5D_X(M_F^2 + 2\sigma_F^2 + \eta^2_{K-1}(M_H^2 + 2\sigma_H^2))}{8K^{3/4}D_X (L_F+\eta_{K-1}L_H)+K^{5/4}\sqrt{M_F^2 + 2\sigma_F^2+\eta_{K-1}(M_H^2+ 2\sigma_H^2)}} + \tfrac{D_X(20 (M_F^2 + 2\sigma_F^2) + \sigma_F^2)}{3K^{1/4}\sqrt{M_F^2+2\sigma_F^2}} +\tfrac{D_X (10(M_H^2 + 2\sigma_H^2) + \sigma_H^2)}{K^{1/2}\sqrt{M_H^2 + 2\sigma_H^2}} 
    \end{aligned}}
\end{equation}
        \begin{equation}\label{eq:adapt-feasib}
    \Scale[.9]{
\begin{aligned}
    &\Ebb[\gap(\bar{x}_K, F, X)]\leq  D_X\bigg[\tfrac{32D_X(L_F + \eta_{K-1}L_H)}{K} + \tfrac{2\sqrt{M_F^2 + 2\sigma_F^2 + \eta^2_{K-1}(M_H^2 + 2\sigma_H^2)}}{K^{1/2}}\bigg]\\
    &
    + 2\tfrac{D_X}{K}[2D_X(2L_F + (\eta_{K-1} + \eta_1)L_H)
     + 2M_F + 4\sigma_F + (\eta_{K-1}+\eta_1)(M_H+2\sigma_H)]\\
     &+ 4\tfrac{D_X(K+2)^{3/4}}{K}[2D_X(L_F +  \eta_1 L_H)
     + M_F + 2\sigma_F + \eta_1(M_H+2\sigma_H)]\\
         &+ \tfrac{D_X(20 (M_F^2 + 2\sigma_F^2) + \sigma_F^2)}{3K^{1/2}\sqrt{M_F^2+2\sigma_F^2}} +\tfrac{D_X (10(M_H^2 + 2\sigma_H^2) + \sigma_H^2)}{K^{3/4}\sqrt{M_H^2 + 2\sigma_H^2}}
        + \tfrac{2C_HD_X}{K^{1/4}}.\\
\end{aligned}
    }
\end{equation}
        Additionally, similar to theorem \ref{thm:lower-bound-mono}, we obtain the following lower bound for optimality gap
        \begin{equation}\label{eq:adapt-lower}
        \begin{aligned}
          & \Ebb[\gap(\bar{x}_K, H, X^*_F)]\geq -\tfrac{B_H}{\alpha}\bigg[D_X\big[\tfrac{32D_X(L_F + \eta_{K-1}L_H)}{K} + \tfrac{2\sqrt{M_F^2 + 2\sigma_F^2 + \eta^2_{K-1}(M_H^2 + 2\sigma_H^2)}}{K^{1/2}}\big]\\
    &
    + 2\tfrac{D_X}{K}[2D_X(2L_F + (\eta_{K-1} + \eta_1)L_H)
     + 2M_F + 4\sigma_F + (\eta_{K-1}+\eta_1)(M_H+2\sigma_H)]\\
     &+ 4\tfrac{D_X(K+2)^{3/4}}{K}[2D_X(L_F +  \eta_1 L_H)
     + M_F + 2\sigma_F + \eta_1(M_H+2\sigma_H)]\\
         &+ \tfrac{D_X(20 (M_F^2 + 2\sigma_F^2) + \sigma_F^2)}{3K^{1/2}\sqrt{M_F^2+2\sigma_F^2}} +\tfrac{D_X (10(M_H^2 + 2\sigma_H^2) + \sigma_H^2)}{K^{3/4}\sqrt{M_H^2 + 2\sigma_H^2}}
        + \tfrac{2C_HD_X}{K^{1/4}}\bigg].
        \end{aligned}
        \end{equation}
    \end{theorem}
    \vspace{-5mm}
    \begin{proof}
        First, it is easy to see that step-size policy in \eqref{eq:adap-step-size} satisfies \eqref{eq:adap-con}. Therefore, similar to \eqref{eq:mono-optim-feasib-1} and multiplying both sides of \eqref{eq:mono-optim-feasib-1} we obtain the following 
        \begin{equation}\label{eq:adapt-mono-optim-feasib-1}
\Scale[0.9]{
    \begin{aligned}
        &\inprod{H(x_F^*)}{\tau_k(x_{k+1}-x^*_F)}\leq \tfrac{\tau_k}{2\gamma_k\eta_k}[\gnorm{x^*_F-x_{k}}{}{2}-\gnorm{x^*_F-x_{k+1}}{}{2} -  \gnorm{x_{k+1}-x_{k}}{}{2}]\\
        &+ \tfrac{\tau_k}{\eta_k}\inprod{\Delta \Ofrak_{k+1}}{x_{k+1}-x^*_F} - \tfrac{\tau_k\theta_k}{\eta_k}\inprod{\Delta \Ofrak_k}{x_k-x^*_F}
         + \tfrac{\tau_k\theta_k}{\eta_k}(L_F + \eta_{k-1} L_H)\gnorm{x_k-x_{k-1}}{}{}\gnorm{x_{k+1}-x_{k}}{}{} \\
         &+ \tfrac{\tau_k\theta_k}{\eta_k}(M_F + \eta_{k-1}M_H)\gnorm{x_{k+1}-x_k}{}{}\\
         & -\tfrac{\tau_k}{\eta_k}\inprod{\delta_{k+1}^F + \eta_k \delta_{k+1}^H}{x_{k+1}-x^*_F} - \tfrac{\tau_k\theta_k}{\eta_k}\inprod{\delta_k^F -\delta_{k-1}^F+ \eta_{k-1}[\delta_k^H -\delta_{k-1}^H]}{x_{k+1}-x_k}.
    \end{aligned}}
\end{equation}
Next, considering $-\tfrac{\tau_k}{\eta_k}\inprod{\delta_{k+1}^F + \eta_k \delta_{k+1}^H}{x_{k+1}-x^*_F}$, one can rewrite it as below 
\begin{equation}\label{eq:adapt-expectation-1}
\Scale[1]{
\begin{aligned}
    - \tfrac{\tau_k}{\eta_k}\inprod{\delta_{k+1}^F + \eta_k \delta_{k+1}^H}{x_{k+1}-x^*_F} = -&\tfrac{\tau_k}{\eta_k}\inprod{\delta_{k+1}^F + \eta_k \delta_{k+1}^H}{x_{k+1}^a-x^*_F} \\&-\tfrac{\tau_k}{\eta_k}\inprod{\delta_{k+1}^F + \eta_k \delta_{k+1}^H}{x_{k+1}-x_{k+1}^a},
\end{aligned}}
\end{equation}
Moreover, we add and subtract $\tfrac{\tau_{k-1}}{2\gamma_{k-1}\eta_{k-1}}\gnorm{x_F^*-x_k}{}{}$ to \eqref{eq:adapt-mono-optim-feasib-1}. Thus, we have the following 
\begin{equation}\label{eq:adapt-mono-optim-feasib-2}
\Scale[0.9]{
    \begin{aligned}
        &\inprod{H(x_F^*)}{\tau_k(x_{k+1}-x^*_F)}\leq \tfrac{\tau_{k-1}}{2\gamma_{k-1}\eta_{k-1}}\gnorm{x^*_F-x_{k}}{}{2}-\tfrac{\tau_k}{2\gamma_k\eta_k}\gnorm{x^*_F-x_{k+1}}{}{2} -  \tfrac{\tau_k}{2\gamma_k\eta_k} \gnorm{x_{k+1}-x_{k}}{}{2}\\
        &+ \tfrac{\tau_k}{\eta_k}\inprod{\Delta \Ofrak_{k+1}}{x_{k+1}-x^*_F} - \tfrac{\tau_k\theta_k}{\eta_k}\inprod{\Delta \Ofrak_k}{x_k-x^*_F}
         + \tfrac{\tau_k\theta_k}{\eta_k}(L_F + \eta_{k-1} L_H)\gnorm{x_k-x_{k-1}}{}{}\gnorm{x_{k+1}-x_{k}}{}{} \\
         &+ \tfrac{\tau_k\theta_k}{\eta_k}(M_F + \eta_{k-1}M_H)\gnorm{x_{k+1}-x_k}{}{}- \tfrac{\tau_k\theta_k}{\eta_k}\inprod{\delta_k^F -\delta_{k-1}^F+ \eta_{k-1}[\delta_k^H -\delta_{k-1}^H]}{x_{k+1}-x_k}\\
         &  -\tfrac{\tau_k}{\eta_k}\inprod{\delta_{k+1}^F + \eta_k \delta_{k+1}^H}{x_{k+1}^a-x^*_F} -\tfrac{\tau_k}{\eta_k}\inprod{\delta_{k+1}^F + \eta_k \delta_{k+1}^H}{x_{k+1}-x_{k+1}^a}\\ 
         & + [\tfrac{\tau_{k}}{2\gamma_{k}\eta_{k}}-\tfrac{\tau_{k-1}}{2\gamma_{k-1}\eta_{k-1}}]\gnorm{x^*_F-x_{k}}{}{2}
    \end{aligned}}
\end{equation}
Moreover, regarding \eqref{eq:recur-augmented} in Lemma \ref{lemma:aug-seq} and by adding and subtracting $\tfrac{\tau_{k-1}}{2\gamma_{k-1}\eta_{k-1}}\gnorm{x_F^*-x_{k+1}^a}{}{}$ to right-hand side of \eqref{eq:recur-augmented}, one can rewrite \eqref{eq:adapt-mono-optim-feasib-2} as follows 
\begin{equation}\label{eq:adapt-mono-optim-feasib-3}
\Scale[0.9]{
    \begin{aligned}
        &\inprod{H(x_F^*)}{\tau_k(x_{k+1}-x^*_F)}\leq \tfrac{\tau_{k-1}}{2\gamma_{k-1}\eta_{k-1}}\gnorm{x^*_F-x_{k}}{}{2}-\tfrac{\tau_k}{2\gamma_k\eta_k}\gnorm{x^*_F-x_{k+1}}{}{2} -  \tfrac{\tau_k}{2\gamma_k\eta_k} \gnorm{x_{k+1}-x_{k}}{}{2}\\
        &+ \tfrac{\tau_k}{\eta_k}\inprod{\Delta \Ofrak_{k+1}}{x_{k+1}-x^*_F} - \tfrac{\tau_k\theta_k}{\eta_k}\inprod{\Delta \Ofrak_k}{x_k-x^*_F}
         + \tfrac{\tau_k\theta_k}{\eta_k}(L_F + \eta_{k-1} L_H)\gnorm{x_k-x_{k-1}}{}{}\gnorm{x_{k+1}-x_{k}}{}{} \\
         &+ \tfrac{\tau_k\theta_k}{\eta_k}(M_F + \eta_{k-1}M_H)\gnorm{x_{k+1}-x_k}{}{}- \tfrac{\tau_k\theta_k}{\eta_k}\inprod{\delta_k^F -\delta_{k-1}^F+ \eta_{k-1}[\delta_k^H -\delta_{k-1}^H]}{x_{k+1}-x_k}\\
         & -\tfrac{\tau_k}{\eta_k}\inprod{\delta_{k+1}^F + \eta_k \delta_{k+1}^H}{x_{k+1}-x_{k+1}^a}
         + [\tfrac{\tau_{k}}{2\gamma_{k}\eta_{k}}-\tfrac{\tau_{k-1}}{2\gamma_{k-1}\eta_{k-1}}]\gnorm{x^*_F-x_{k}}{}{2} + [\tfrac{\tau_{k}}{2\gamma_{k}\eta_{k}}-\tfrac{\tau_{k-1}}{2\gamma_{k-1}\eta_{k-1}}]\gnorm{x^*_F-x_{k+1}^a}{}{2} \\
         &  + \tfrac{\tau_{k-1}}{2\gamma_{k-1}\eta_{k-1}}\gnorm{x^*_F-x_{k+1}^a}{}{2}  -  \tfrac{\tau_{k}}{2\gamma_{k}\eta_{k}} \gnorm{x^*_F-x_{k+2}^a}{}{2} + \tfrac{\tau_k\gamma_k}{2\eta_k} \gnorm{\delta_{k+1}^F+\eta_k\delta_{k+1}^H}{}{2},
    \end{aligned}}
\end{equation}
Now, let us sum \eqref{eq:adapt-mono-optim-feasib-3} from $k=1$ to $K-1$. We have 
\begin{equation*}\label{eq:adapt-mono-optim-feasib-4}
\Scale[1]{
    \begin{aligned}
        &\tsum_{k=1}^{K-1}\inprod{H(x_F^*)}{\tau_k(x_{k+1}-x^*_F)}\leq \tfrac{\tau_{0}}{2\gamma_{0}\eta_{0}}\gnorm{x^*_F-x_{1}}{}{2}-\tfrac{\tau_{K-1}}{2\gamma_{K-1}\eta_{K-1}}\gnorm{x^*_F-x_{K}}{}{2}
        + \tfrac{\tau_{0}}{2\gamma_{0}\eta_{0}}\gnorm{x^*_F-x_{2}^a}{}{2}\\
        &-  \tfrac{\tau_{K-1}}{2\gamma_{K-1}\eta_{K-1}} \gnorm{x^*_F-x_{K+1}^a}{}{2}
        + \tsum_{k=1}^{K-1}\tfrac{\tau_k}{\eta_k}\inprod{\Delta \Ofrak_{k+1}}{x_{k+1}-x^*_F} - \tfrac{\tau_k\theta_k}{\eta_k}\inprod{\Delta \Ofrak_k}{x_k-x^*_F}\\
        &
         + \tsum_{k=1}^{K-1}[\tfrac{\tau_k\theta_k}{\eta_k}(L_F + \eta_{k-1} L_H)\gnorm{x_k-x_{k-1}}{}{}\gnorm{x_{k+1}-x_{k}}{}{} -\tfrac{\tau_k}{10\gamma_k\eta_k}\gnorm{x_{k+1}-x_k}{}{2}\\
         &-\tfrac{\tau_{k-1}}{10\gamma_{k-1}\eta_{k-1}}\gnorm{x_{k}-x_{k-1}}{}{2}] 
         + \tsum_{k=1}^{K-1}\tfrac{\tau_k\theta_k}{\eta_k}(M_F + \eta_{k-1}M_H)\gnorm{x_{k+1}-x_k}{}{}-\tfrac{\tau_k}{10\gamma_k\eta_k}\gnorm{x_{k+1}-x_k}{}{2}\\
         & -  \tsum_{k=1}^{K-1}\tfrac{\tau_k\theta_k}{\eta_k}\inprod{\delta_k^F -\delta_{k-1}^F}{x_{k+1}-x_k}-\tfrac{\tau_k}{10\gamma_k\eta_k} \gnorm{x_{k+1}-x_{k}}{}{2}\\
         & - \tsum_{k=1}^{K-1}\tfrac{\tau_k\theta_k\eta_{k-1}}{\eta_k}\inprod{\delta_k^H -\delta_{k-1}^H}{x_{k+1}-x_k}-\tfrac{\tau_k}{10\gamma_k\eta_k} \gnorm{x_{k+1}-x_{k}}{}{2}\\
         & -\tsum_{k=1}^{K-1}\tfrac{\tau_k}{\eta_k}\inprod{\delta_{k+1}^F + \eta_k \delta_{k+1}^H}{x_{k+1}-x_{k+1}^a}\\
         &
         + [\tfrac{2\tau_{K-1}}{\gamma_{K-1}\eta_{K-1}}-\tfrac{2\tau_{0}}{\gamma_{0}\eta_{0}}](D_X^2 +D_{X_F^*}^2)\\
         &  +\tsum_{k=1}^{K-1} \tfrac{\tau_k\gamma_k}{2\eta_k} \gnorm{\delta_{k+1}^F+\eta_k\delta_{k+1}^H}{}{2}.
    \end{aligned}}
\end{equation*}
Note that 
\vspace{-3mm}
\begin{equation}\label{eq:adapt_DX}
    \tfrac{\tau_0}{2\gamma_0\eta_0}\gnorm{x_F^* -x_1}{}{2}\leq \tfrac{2\tau_0}{\gamma_0\eta_0}D_X^2,\quad \tfrac{\tau_0}{2\gamma_0\eta_0}\gnorm{x_F^* -x_2^a}{}{2}\leq \tfrac{2\tau_0}{\gamma_0\eta_0}D_{X_F^*}^2
\end{equation}
Then, considering conditions \eqref{eq:adap-con} and relations in \eqref{eq:non-smooth-youngs}, one obtains the following 
\begin{equation}\label{eq:adapt-mono-optim-feasib-5}
\Scale[0.9]{
    \begin{aligned}
        &\tsum_{k=1}^{K-1}\inprod{H(x_F^*)}{\tau_k(x_{k+1}-x^*_F)}\leq [\tfrac{2\tau_{K-1}}{\gamma_{K-1}\eta_{K-1}}](D_X^2 +D_{X_F^*}^2) + \tfrac{\tau_{K-1}}{\eta_{K-1}}\inprod{\Delta \Ofrak_{K}}{x_{K}-x^*_F}\\
         &+ \tsum_{k=1}^{K-1} \tfrac{5\theta^2\gamma_k\tau_k}{\eta_k}(M_F^2+\gnorm{\delta_k^F -\delta_{k-1}^F}{}{2}+\eta_{k-1}^2(M_H^2+\gnorm{\delta_k^F -\delta_{k-1}^F}{}{2}))
        \\
         &+\tsum_{k=1}^{K-1} \tfrac{\tau_k\gamma_k}{2\eta_k} \gnorm{\delta_{k+1}^F+\eta_k\delta_{k+1}^H}{}{2} -\tsum_{k=1}^{K-1}\tfrac{\tau_k}{\eta_k}\inprod{\delta_{k+1}^F + \eta_k \delta_{k+1}^H}{x_{k+1}-x_{k+1}^a}.
    \end{aligned}}
\end{equation}
Further, considering \eqref{eq:tower-propert} and using the similar arguments we used to obtain \eqref{eq:mono-optim-feasib-3-1}, we have the following bound for the optimality 
\begin{equation}\label{eq:adapt-mono-optim-feasib-6}
\Scale[0.9]{
    \begin{aligned}
        &\Ebb[\inprod{H(x_F^*)}{\bar{x}_K-x^*_F}]\leq [\tfrac{2\tau_{K-1}}{K\gamma_{K-1}\eta_{K-1}}](D_X^2 +D_{X_F^*}^2) + \tfrac{5\gamma_{K-1}}{K\eta_{K-1}}(M_F^2+2\sigma_F^2+\eta_{K-1}^2(M_H^2+2\sigma_H^2))\\
         &+ \tsum_{k=1}^{K-1} \tfrac{5\theta^2\gamma_k\tau_k}{K\eta_k}(M_F^2+2\sigma_F^2+\eta_{k-1}^2(M_H^2+2\sigma_H^2))
        +\tsum_{k=1}^{K-1} \tfrac{\tau_k\gamma_k}{2K\eta_k} (\sigma_F^2 + \eta_k\sigma_H^2).
    \end{aligned}}
\end{equation}
Using the updating rule in \eqref{eq:adap-step-size} we obtain the following 
\begin{equation}\label{eq:adapt-mono-optim-feasib-7}
\Scale[0.9]{
    \begin{aligned}
        &\Ebb[\inprod{H(x_F^*)}{\bar{x}_K-x^*_F}]\leq\tfrac{D_X^2 +D_{X_F^*}^2}{D_X}\bigg[\tfrac{16D_X(L_F + \eta_{K-1}L_H)}{K^{3/4}} + \tfrac{2\sqrt{M_F^2 + 2\sigma_F^2 + \eta^2_{K-1}(M_H^2 + 2\sigma_H^2)}}{K^{1/4}}\bigg]\\
        & + \tfrac{5D_X(M_F^2 + 2\sigma_F^2 + \eta^2_{K-1}(M_H^2 + 2\sigma_H^2))}{8K^{3/4}D_X (L_F+\eta_{K-1}L_H)+K^{5/4}\sqrt{M_F^2 + 2\sigma_F^2+\eta_{K-1}(M_H^2+ 2\sigma_H^2)}} + \tfrac{D_X(20 (M_F^2 + 2\sigma_F^2) + \sigma_F^2)}{3K^{1/4}\sqrt{M_F^2+2\sigma_F^2}} +\tfrac{D_X (10(M_H^2 + 2\sigma_H^2) + \sigma_H^2)}{K^{1/2}\sqrt{M_H^2 + 2\sigma_H^2}}.
    \end{aligned}}
\end{equation}
Finally, taking the maximum concerning the set $X^*_F$ and the definition of the optimality gap function, we obtain \eqref{eq:adapt-optim-mono}. The left-hand side of optimality uses the same argument we mentioned in Theorem \ref{thm:convex-case-optim-feasib}. Now, we move to the feasibility gap. Considering Lemma \ref{lem:one-iteration-iropex} and using the same procedure we used to get \eqref{eq:adapt-mono-optim-feasib-3}, we have
\begin{equation}\label{eq:adapt-mono-optim-feasib-8}
\Scale[0.9]{
    \begin{aligned}
        &\inprod{F(x)}{\tau_k(x_{k+1}-x)}\leq \tfrac{\tau_{k-1}}{2\gamma_{k-1}}\gnorm{x-x_{k}}{}{2}-\tfrac{\tau_k}{2\gamma_k}\gnorm{x-x_{k+1}}{}{2} -  \tfrac{\tau_k}{2\gamma_k} \gnorm{x_{k+1}-x_{k}}{}{2}\\
        &+ \tau_k\inprod{\Delta \Ofrak_{k+1}}{x_{k+1}-x} - \tau_k\theta_k\inprod{\Delta \Ofrak_k}{x_k-x}
         + \tau_k\theta_k(L_F + \eta_{k-1} L_H)\gnorm{x_k-x_{k-1}}{}{}\gnorm{x_{k+1}-x_{k}}{}{} \\
         &+ \tau_k\theta_k(M_F + \eta_{k-1}M_H)\gnorm{x_{k+1}-x_k}{}{}- \tau_k\theta_k\inprod{\delta_k^F -\delta_{k-1}^F+ \eta_{k-1}[\delta_k^H -\delta_{k-1}^H]}{x_{k+1}-x_k}\\
         & -\tau_k\inprod{\delta_{k+1}^F + \eta_k \delta_{k+1}^H}{x_{k+1}-x_{k+1}^a}
         + [\tfrac{\tau_{k}}{2\gamma_{k}}-\tfrac{\tau_{k-1}}{2\gamma_{k-1}}]\gnorm{x-x_{k}}{}{2} + [\tfrac{\tau_{k}}{2\gamma_{k}}-\tfrac{\tau_{k-1}}{2\gamma_{k-1}}]\gnorm{x-x_{k+1}^a}{}{2} \\
         &  + \tfrac{\tau_{k-1}}{2\gamma_{k-1}}\gnorm{x-x_{k+1}^a}{}{2}  -  \tfrac{\tau_{k}}{2\gamma_{k}} \gnorm{x-x_{k+2}^a}{}{2} + \tfrac{\tau_k\gamma_k}{2} \gnorm{\delta_{k+1}^F+\eta_k\delta_{k+1}^H}{}{2} + 2 C_H D_X\eta_k,
    \end{aligned}}
\end{equation}
Summing \eqref{eq:adapt-mono-optim-feasib-8} from $k=1$ to $K-1$, in view of \eqref{eq:adapt_DX} and using the same argument we used in \eqref{eq:adapt-mono-optim-feasib-5} we have  
\begin{equation}\label{eq:adapt-mono-optim-feasib-9}
    \Scale[.9]{
\begin{aligned}
    &\tsum_{k=1}^{K-1}\inprod{F(x)}{\tau_k(x_{k+1}-x)}\leq  \tfrac{4\tau_{K-1}}{\gamma_{K-1}}D_X^2
    + \tsum_{k=1}^{K-1}\tau_k\inprod{\Delta \Ofrak_{k+1}}{x_{k+1}-x} - \tau_k\theta_k\inprod{\Delta \Ofrak_k}{x_k-x}\\
         &+ \tsum_{k=1}^{K-1} 5\theta^2\gamma_k\tau_k(M_F^2+\gnorm{\delta_k^F -\delta_{k-1}^F}{}{2}+\eta_{k-1}^2(M_H^2+\gnorm{\delta_k^F -\delta_{k-1}^F}{}{2}))
        \\
         &+\tsum_{k=1}^{K-1} \tfrac{\tau_k\gamma_k}{2} \gnorm{\delta_{k+1}^F+\eta_k\delta_{k+1}^H}{}{2} + \tsum_{k=1}^{K-1} 2C_HD_X\eta_k -\tsum_{k=1}^{K-1}\tfrac{\tau_k}{\eta_k}\inprod{\delta_{k+1}^F + \eta_k \delta_{k+1}^H}{x_{k+1}-x_{k+1}^a}.
\end{aligned}
    }
\end{equation}
Moreover, we need to simplify $\tsum_{k=1}^{K-1}\tau_k\inprod{\Delta \Ofrak_{k+1}}{x_{k+1}-x} - \tau_k\theta_k\inprod{\Delta \Ofrak_k}{x_k-x}$ as below 
\begin{equation}\label{eq:adapt-telescope}
\begin{aligned}
     \tsum_{k=1}^{K-1}\tau_k\inprod{\Delta \Ofrak_{k+1}}{x_{k+1}-x} &- \tau_k\theta_k\inprod{\Delta \Ofrak_k}{x_k-x} = \tau_{K-1}\inprod{\Delta \Ofrak_{K}}{x_{K}-x}\\
     &+ \tsum_{k=1}^{K-1}(\tau_k - \tau_{k+1}\theta_{k+1})\inprod{\Delta \Ofrak_{k+1}}{x_{k+1}-x} - \tau_{1}\inprod{\Delta \Ofrak_{1}}{x_{1}-x}
\end{aligned}
\end{equation}
Further, from \eqref{eq:adapt-telescope} and \eqref{eq:adap-step-size}, one can bound $\Ebb[\tsum_{k=1}^{K-1}\tau_k\inprod{\Delta \Ofrak_{k+1}}{x_{k+1}-x} - \tau_k\theta_k\inprod{\Delta \Ofrak_k}{x_k-x}] $  as follows 
\begin{equation}\label{eq:adapt-telescope-epectation}
    \begin{aligned}
     \Ebb&[\tsum_{k=1}^{K-1}\tau_k\inprod{\Delta \Ofrak_{k+1}}{x_{k+1}-x} - \tau_k\theta_k\inprod{\Delta \Ofrak_k}{x_k-x}]\leq 2D_X[2D_X(2L_F + (\eta_{K-1} + \eta_1)L_H)\\
     &+ 2M_F + 4\sigma_F + (\eta_{K-1}+\eta_1)(M_H+2\sigma_H)]\\
     &+ 2D_X[2D_X(L_F +  \eta_1 L_H)
     + M_F + 2\sigma_F + \eta_1(M_H+2\sigma_H)]\tsum_{k=1}^{K-1}(1-(\tfrac{k+1}{k+2})^{1/4})
\end{aligned}
\end{equation}
Note that $\tsum_{k=1}^{K-1}(1-(\tfrac{k+1}{k+2})^{1/4}) =\tsum_{k=1}^{K-1}(1-(1-\tfrac{1}{k+2})^{1/4}) \leq 2(K+2)^{3/4} $. Therefore we can rewrite \eqref{eq:adapt-telescope-epectation} below 
\begin{equation}\label{eq:adapt-telescope-epectation-2}
    \begin{aligned}
     \Ebb&[\tsum_{k=1}^{K-1}\tau_k\inprod{\Delta \Ofrak_{k+1}}{x_{k+1}-x} - \tau_k\theta_k\inprod{\Delta \Ofrak_k}{x_k-x}]\leq 2D_X[2D_X(2L_F + (\eta_{K-1} + \eta_1)L_H)\\
     &+ 2M_F + 4\sigma_F + (\eta_{K-1}+\eta_1)(M_H+2\sigma_H)]
     + 2D_X[2D_X(L_F +  \eta_1 L_H)\\
     &+ M_F + 2\sigma_F + \eta_1(M_H+2\sigma_H)]2(K+2)^{3/4}.
\end{aligned}
\end{equation}
By an analogous reasoning we used to obtain \eqref{eq:adapt-mono-optim-feasib-6}, we have
\begin{equation}\label{eq:adapt-mono-optim-feasib-11}
    \Scale[1]{
\begin{aligned}
    &\Ebb[\inprod{F(x)}{\bar{x}_K-x}]\leq\tfrac{4\tau_{K-1}}{K\gamma_{K-1}}D_X^2
    + 2\tfrac{D_X}{K}[2D_X(2L_F + (\eta_{K-1} + \eta_1)L_H)\\
     &+ 2M_F + 4\sigma_F + (\eta_{K-1}+\eta_1)(M_H+2\sigma_H)]\\
     &+ 4\tfrac{D_X(K+2)^{3/4}}{K}[2D_X(L_F +  \eta_1 L_H)
     + M_F + 2\sigma_F + \eta_1(M_H+2\sigma_H)]\\
         &+ \tsum_{k=1}^{K-1} \tfrac{5\theta^2\gamma_k\tau_k}{K}(M_F^2+2\sigma_F^2+\eta_{k-1}^2(M_H^2+2\sigma_H^2))
        \\
         &+\tsum_{k=1}^{K-1} \tfrac{\tau_k\gamma_k}{2K} (\sigma_F^2 + \eta_k \sigma_H^2) + \tsum_{k=1}^{K-1} \tfrac{2C_HD_X\eta_k }{K}.
\end{aligned}
    }
\end{equation}
Using the step-size policy in \eqref{eq:adap-step-size} and taking the maximum from both sides regarding set $X$ of \eqref{eq:adapt-mono-optim-feasib-11}, we obtain the following upper bound
\begin{equation*}\label{eq:adapt-mono-optim-feasib-12}
    \Scale[.9]{
\begin{aligned}
    &\Ebb[\gap(\bar{x}_K, F, X)]\leq  D_X\bigg[\tfrac{32D_X(L_F + \eta_{K-1}L_H)}{K} + \tfrac{2\sqrt{M_F^2 + 2\sigma_F^2 + \eta^2_{K-1}(M_H^2 + 2\sigma_H^2)}}{K^{1/2}}\bigg]\\
    &
    + 2\tfrac{D_X}{K}[2D_X(2L_F + (\eta_{K-1} + \eta_1)L_H)
     + 2M_F + 4\sigma_F + (\eta_{K-1}+\eta_1)(M_H+2\sigma_H)]\\
     &+ 4\tfrac{D_X(K+2)^{3/4}}{K}[2D_X(L_F +  \eta_1 L_H)
     + M_F + 2\sigma_F + \eta_1(M_H+2\sigma_H)]\\
         &+ \tfrac{D_X(20 (M_F^2 + 2\sigma_F^2) + \sigma_F^2)}{3K^{1/2}\sqrt{M_F^2+2\sigma_F^2}} +\tfrac{D_X (10(M_H^2 + 2\sigma_H^2) + \sigma_H^2)}{K^{3/4}\sqrt{M_H^2 + 2\sigma_H^2}}
        + \tfrac{2C_HD_X}{K^{1/4}}.\\
\end{aligned}
    }
\end{equation*}
The lower bound for the optimality gap is obtained through the same line of argument we used in Theorem \ref{thm:lower-bound-mono}. 
    \end{proof}
    Note that the convergence results for the strongly monotone case are the same as Remark \ref{rem:strongly-mono} and we do not mention it again for brevity. 
\subsection{Convergence analysis for the smooth VI}\label{sec:smooth-VI-appden}
 \begin{theorem}\label{thm:convex-case-optim-feasib-smooth-stoch-F}
         Let us assume that problem \eqref{eq:bilevel-vi} is monotone $(\mu_H = 0)$. Additionally, let $M_F = 0$, and suppose we have the following step-size policy with mini-batching of size $B = K$
         \vspace{-2mm}
         \begin{equation}\label{eq:step-size-con-bivel-smooth-stoch-F}
         \Scale[.95]{
             \tau_k =\tau =1,\quad \theta_k =\theta=1,\quad \eta_k= \eta= K^{-\tfrac{1}{2}},\quad \gamma_k = \gamma =\tfrac{D_X}{8D_X(L_F +\eta L_H)+ \sqrt{M_H^2+2(\sigma_H^2 + \sigma_F^2)}},}
         \end{equation}
         Then, Algorithm \ref{alg:IRopex} gives the following optimality and feasibility gaps for $K\geq 1$
         \begin{equation}\label{eq:mono-optim-smooth-stoch-2}
\Scale[1]{
\begin{aligned}
    -B_H\Ebb[\text{dist}(\bar{x}_K,X^*_F)]\leq \Ebb[\gap(\bar{x}_K, H, X^*_F)] \leq  & D_X\bigg[\big(16D_X(\tfrac{L_F}{K^{1/2}} +\tfrac{ L_H}{K}) + \tfrac{2\sqrt{M_H^2+ 2(\sigma_H^2 + \sigma_F^2)}}{K^{1/2}} \big)\\
    &+\tfrac{5(M_H^2+2(\sigma_H^2+\sigma_F^2)) + \sigma_H^2+\sigma_F^2}{8D_XK^{1/2}(L_F + \eta L_H) + K^{1/2}\sqrt{M_H^2+2(\sigma_H^2+\sigma_F^2)}}\\
         &+ \tfrac{5(M_H^2+2(\sigma_H^2+\sigma_F^2))}{8D_XK^{3/2}(L_F + \eta L_H) + K^{3/2}\sqrt{M_H^2+2(\sigma_H^2+\sigma_F^2)}}\bigg].
     \end{aligned}}
\end{equation}
\vspace{-3mm}
and 
\begin{equation}\label{eq:mono-feasib-smooth-stoch-F}
\Scale[.9]{
    \begin{aligned}
        \Ebb[\gap(\bar{x}_K,F,X)]\leq &  D_X\bigg[\big(16D_X(\tfrac{L_F}{K} +\tfrac{ L_H}{K^{3/2}}) + \tfrac{\sqrt{M_H^2+2(\sigma_H^2+\sigma_F^2)}}{K} \big)+\tfrac{5(M_H^2+2(\sigma_H^2+\sigma_F^2)) + \sigma_H^2+\sigma_F^2}{8D_XK(L_F + \eta L_H) + K\sqrt{M_H^2+2(\sigma_H^2+\sigma_F^2)}}
         \\
         &+ \tfrac{5(M_H^2+2(\sigma_H^2+\sigma_F^2))}{8D_XK^{2}(L_F + \eta L_H) + K^{2}\sqrt{M_H^2+2(\sigma_H^2+\sigma_F^2)}}+\tfrac{2 C_H}{K^{1/2}}].
    \end{aligned}}
\end{equation}
\end{theorem}
\vspace{-3mm}
\begin{proof}
    The proof stems from the same reasoning we used in Theorem \ref{thm:convex-case-optim-feasib}.
\end{proof}
 \begin{theorem}\label{thm:convex-case-optim-feasib-smooth}
         Let us assume problem \eqref{eq:bilevel-vi} is monotone $(\mu_H = 0)$ and $H_F = \sigma_F =0$. Suppose we have the following step-size policy 
         \begin{equation}\label{eq:step-size-con-bivel-smooth}
         \Scale[.95]{
             \tau_k =\tau =1,\quad \theta_k =\theta=1,\quad \eta_k= \eta= K^{-\tfrac{1}{2}},\quad \gamma_k = \gamma =\tfrac{D_X}{8D_X(L_F +\eta L_H)+ \sqrt{M_H^2+2\sigma_H^2}},}
         \end{equation}
         Then, Algorithm \ref{alg:IRopex} gives the following optimality and feasibility gaps for $K\geq 1$
         \begin{equation}\label{eq:mono-optim-smooth}
\Scale[1]{
\begin{aligned}
    -B_H\Ebb[\text{dist}(\bar{x}_K,X^*_F)]\leq \Ebb[\gap(\bar{x}_K, H, X^*_F)] \leq  & D_X\bigg[\big(16D_X(\tfrac{L_F}{K^{1/2}} +\tfrac{ L_H}{K}) + \tfrac{2\sqrt{M_H^2+ 2\sigma_H^2 }}{K^{1/2}} \big)\\
    &+\tfrac{5(M_H^2+2\sigma_H^2) + \sigma_H^2}{8D_XK^{1/2}(L_F + \eta L_H) + K^{1/2}\sqrt{M_H^2+2\sigma_H^2}}\\
         &+ \tfrac{5(M_H^2+2\sigma_H^2)}{8D_XK^{3/2}(L_F + \eta L_H) + K^{3/2}\sqrt{M_H^2+2\sigma_H^2}}\bigg].
     \end{aligned}}
\end{equation}
and 
\begin{equation}\label{eq:mono-feasib-smooth}
\Scale[.9]{
    \begin{aligned}
        \Ebb[\gap(\bar{x}_K,F,X)]\leq &  D_X\bigg[\big(16D_X(\tfrac{L_F}{K} +\tfrac{ L_H}{K^{3/2}}) + \tfrac{\sqrt{M_H^2+2\sigma_H^2}}{K} \big)+\tfrac{5(M_H^2+2\sigma_H^2) + \sigma_H^2}{8D_XK(L_F + \eta L_H) + K\sqrt{M_H^2+2\sigma_H^2}}
         \\
         &+ \tfrac{5(M_H^2+2\sigma_H^2)}{8D_XK^{2}(L_F + \eta L_H) + K^{2}\sqrt{M_H^2+2\sigma_H^2}}+\tfrac{2 C_H}{K^{1/2}}].
    \end{aligned}}
\end{equation}
Moreover, considering the following policy in the strongly monotone case $(\mu_H>0)$, 
\begin{equation}\label{eq:step-size-con-bivel-smooth-strong}
         \Scale[.95]{
             \tau_k  =k+1,\quad \theta_k =\tfrac{k}{k+1},\quad \eta_k= \eta= K^{-\tfrac{1}{2}},\quad \gamma_k = \gamma =\tfrac{D_X}{8D_X(L_F +\eta L_H)+ \sqrt{M_H^2+2\sigma_H^2}},}
         \end{equation}
         Then, for $K\geq \tfrac{1}{2\gamma_k\eta_{k-1}\mu_H}$, Algorithm \ref{alg:IRopex} gives the following optimality and feasibility gaps
         \begin{equation}\label{eq:mono-optim-smooth-strong}
\Scale[1]{
\begin{aligned}
   -B_H\Ebb[\text{dist}(\bar{x}_K,X^*_F)]\leq \Ebb[\gap(\bar{x}_K, H, X^*_F)] \leq  & D_X\bigg[\big(16D_X(\tfrac{L_F}{K^{3/2}} +\tfrac{ L_H}{K^2}) + \tfrac{2\sqrt{M_H^2+ 2\sigma_H^2 }}{K^{3/2}} \big)\\
   &+\tfrac{5(M_H^2+2\sigma_H^2) + \sigma_H^2}{8D_XK^{3/2}(L_F + \eta L_H) + K^{3/2}\sqrt{M_H^2+2\sigma_H^2}}\\
         &+ \tfrac{5(M_H^2+2\sigma_H^2}{8D_XK^{3/2}(L_F + \eta L_H) + K^{5/2}\sqrt{M_H^2+2\sigma_H^2}}\bigg].
     \end{aligned}}
\end{equation}
\begin{equation}\label{eq:mono-feasib-smooth-strong}
\Scale[.9]{
    \begin{aligned}
        \Ebb[\gap(\bar{x}_K,F,X)]\leq &  D_X\bigg[\big(16D_X(\tfrac{L_F}{K^2} +\tfrac{ L_H}{K^{5/2}}) + \tfrac{\sqrt{M_H^2+2\sigma_H^2}}{K^2} \big)+\tfrac{5(M_H^2+2\sigma_H^2) + \sigma_H^2}{8D_XK^2(L_F + \eta L_H) + K^2\sqrt{M_H^2+2\sigma_H^2}}
         \\
         &+ \tfrac{5(M_H^2+2\sigma_H^2)}{8D_XK^{3}(L_F + \eta L_H) + K^{3}\sqrt{M_H^2+2\sigma_H^2}}+\tfrac{2 C_H}{K^{3/2}}].
    \end{aligned}}
\end{equation}
\end{theorem}
\begin{proof}
    The proof is similar to Theorem \ref{thm:convex-case-optim-feasib} for the monotone case and Theorem \ref{thm:convex-case-optim-feasib-strong} for the strongly monotone case. 
\end{proof}
Theorem \ref{thm:convex-case-optim-feasib-smooth} shows the improvement in convergence rate in terms of $M_H$ and $\sigma_H$ from $\mathcal{O}(K^{-1/4})$ to $\mathcal{O}(K^{-1/2})$ .
Similar to the nonsmooth stochastic case, we provide the following lower bound for the monotone problem with smooth deterministic operator $F$. 
\begin{theorem}\label{thm:monotone-lower-smooth}
    Suppose $\text{VI}(F,X)$ is $\alpha$-weakly sharp. Then we have 
           \begin{equation}\label{eq:lower-bound-mono-smooth}
           \Scale[1]{
           \begin{aligned}
                \Ebb[\gap(\bar{x}_K, H, X^*_F)]& \geq  -\tfrac{B_H}{\alpha}\Big[2D_X^2\big(8(\tfrac{L_F}{K^{1/2}} +\tfrac{ L_H}{K}) + \tfrac{\sqrt{M_H^2+2\sigma_H^2}}{K^{1/2}} \big)\\
                &+\tfrac{5(M_H^2+2\sigma_H^2) + \sigma_H^2}{8K^{1/2}(L_F + \eta L_H) + K^{1/2}\sqrt{M_H^2+2\sigma_H^2}}
         \\
         &+ \tfrac{5(M_H^2+2\sigma_H^2)}{8K^{3/2}(L_F + \eta L_H) + K^{3/2}\sqrt{M_H^2+\sigma_H^2}}+\tfrac{2 C_H D_X}{K^{1/2}}\Big].
           \end{aligned}}
           \end{equation}
\end{theorem}
\begin{proof}
    The proof uses the same arguments used in Theorem \ref{thm:lower-bound-mono}
\end{proof}

%%%%%%%%%%%%%%%%%%%%%%%%%%%%%%%%%%%%%%%%%%%%%%%%%%%%%%%%%%%%

\newpage

\bibliographystyle{plainnat}
\bibliography{ref.bib}

\end{document}